\documentclass[11pt,reqno]{amsart}
\usepackage{amssymb,mathrsfs,graphicx,subfigure, enumerate}
\usepackage{amsmath,amsfonts,amssymb,amscd,amsthm,bbm}
\usepackage{graphicx,colortbl,here}

\usepackage{comment}
\usepackage{enumitem}
\usepackage{kotex}

\usepackage{stmaryrd}
\usepackage{mathtools}

\newcommand{\Supp}{\operatorname{supp}}
\newcommand{\diam}{\operatorname{diam}}

\newcommand{\Id}{\operatorname{Id}}

\newtheorem{theorem}{Theorem}[section]
\newtheorem{lemma}{Lemma}[section]
\newtheorem{corollary}{Corollary}[section]
\newtheorem{proposition}{Proposition}[section]
\newtheorem{remark}{Remark}[section]
\newtheorem{definition}{Definition}[section]

\title[Periodic solutions of the nonlocal continuity equation]{Existence of periodic measure-valued solutions to the nonlocal continuity equation via optimal transport}

\author[S.-Y. Ha]{Seung-Yeal Ha}
\address[Seung-Yeal Ha]{\newline Department of Mathematical Sciences and Research Institute of Mathematics \newline Seoul National University, Seoul 08826, Republic of Korea} \email{syha@snu.ac.kr}

\author[G. Hwang]{Gyuyoung Hwang}
\address[Gyuyoung Hwang]{\newline Biomedical Mathematics
Group, Pioneer Research Center for Mathematical and Computational Sciences, Institute for Basic Science,
Daejeon 34126, Republic of Korea} \email{hgy0407@snu.ac.kr}

\author[P. Thieullen]{Philippe Thieullen}
\address[Philippe Thieullen]{\newline Institut de Math\'ematiques de Bordeaux \newline
Universit\'e de Bordeaux, 351 cours de la Lib\'eration, F 33405 Talence, France}
\email{philippe.thieullen@u-bordeaux.fr}

\author[J. Yoon]{Jaeyoung Yoon}
\address[Jaeyoung Yoon]{\newline Department of Mathematics, School of Computation, Information and Technology \newline
Technical University of Munich, Boltzmannstrasse 3, 85748 Garching bei M\"unchen, Germany} \email{wodud1516@gmail.com / jaeyoung.yoon@tum.de}

\thanks{\textbf{Acknowledgment.} The work of S.-Y. Ha is supported by National Research Foundation(NRF) grant funded by the Korea government(MIST) (RS-2025-00514472) and the work of P. Thieullen is supported  the CNRS project, Internation Research Netwok (IRN). J. Yoon would like to thank the Alexander von Humboldt Stiftung for support via a postdoctoral research fellowship.}
\date{\today}

\begin{document}

\maketitle

\begin{abstract}
We investigate the existence of periodic solutions for a class of nonlocal continuity equations, which include mean-field equations derived from systems of coupled oscillators. While periodic solutions at the particle level have been studied through the construction of a Poincaré map on a section of an invariant set, extending this analysis to the level of continuity equations presents nontrivial challenges. In particular, setting an appropriate topology for the infinite-dimensional space to show invariance and apply the fixed point argument is not easy. To overcome this difficulty, we use fixed point theorem for geodesically convex spaces constructed by optimal transportation. Specifically, from the disintegration with respect to stationary variable, we define a metric using the Wasserstein-$2$ distance over one-dimensional space, which yields a $CAT(0)$ space. In this topology, we construct an invariance set of probability measures and prove the existence of the periodic measure-valued solution from Schauder's fixed point theorem on geodesically convex spaces. As a corollary, our method directly gives an existence of periodic graph measure solution.
\end{abstract}

\section{Introduction}\label{sec:1}
\setcounter{equation}{0} 
Infinite-dimensional dynamical systems for collective behaviors exhibit a variety of interesting asymptotic behaviors and mathematical challenges. Such infinite-dimensional systems are often described by partial differential equations with nonlocal terms. For example, in fluid dynamics, various aspects of the Hamiltonian structure of the incompressible Euler equations have been studied including the twist property \cite{D-E-J} and the stability of quasi-periodic orbits \cite{B-M}. In plasma physics, Hamiltonian system governed by the Vlasov–Poisson equation exhibits damping \cite{M-V} which is a nontrivial phenomenon arising from the transition to an infinite-dimensional setting. Outside PDE context, Bahsoun et  al \cite{B-L-S} investigated infinite systems of coupled Anosov diffeomorphisms using transfer operator techniques. Recently, dynamical system on the space of measures gained increasing attentions with the connection to the optimal transport \cite{Pic}. In this context, we study an infinite dimensional dynamical system of coupled oscillators in the mean field limit. The oscillators are described by a  measure-valued function $\mu_t(dx,d\omega)$ that encodes at time $t$ the number of oscillators in the volume $dx \times d\omega$ where $x \in\mathbb{R}$ denotes the lifted phase  of a particular oscillator in $\mathbb{S}^1 = \mathbb{R}/2\pi\mathbb{Z}$ and $\omega \in \mathbb{R}$ denotes  its intrinsic frequency (its own virtual frequency if the oscillator were not coupled with the whole population). The distribution $\mu_t$ supported on $\mathbb{R} \times  \mathbb{R}$ is thus a solution of a continuity equation (or a scalar conservation law). Moreover the flux is supposed to be nonlocal as it encodes the coupling between all the oscillators.  The nonlocal continuity equation that we are interested in  appears in the mean-field approximation of interacting particle systems in collective dynamics, such as to name a few,  the kinetic Winfree and Kuramoto models \cite{A-B, Al, A-H-P, A-S, Ku}, the kinetic Cucker-Smale model \cite{H-Liu, H-T}.

In a mesoscopic regime in which the number of oscillators is sufficiently large, statistical description of a large particle system can be effectively approximated by the Vlasov equation via a suitable mean-field limit or graph limit. In this paper, we are interested in periodic measure-valued solutions. In the particular case of a density $\mu_t(dx,d\omega) = f_t(x,\omega) dx \, d\omega$, the periodicity in time is understood in the sense $f_{t+T}(x,\omega) = f_t(x-2\pi,\omega)$ for some period $T>0$ when the frequency of the density is positive. The density $f_t(\cdot,\omega)$ at time $t$ has been lifted to $\mathbb{R}$ in order to distinguish oscillators that rotate much further than others. The  evolution of $\mu_t$ is thus governed by the Cauchy problem of the following continuity equation with  nonlocal flux $F[\mu]$:  
\begin{equation}
\begin{cases} \label{A-1}
\displaystyle \partial_t \mu + \partial_x(F[\mu] \mu) = 0, \quad (t, x, \omega) \in \mathbb R_+ \times \mathbb{R}^2, \\
\displaystyle \mu_0(dx, d\omega) = \mu^{\text{\rm init}}(dx,d\omega),
\end{cases}
\end{equation}
where $\mu^{\text{\rm init}}$ denotes the distribution of the initial phases $x$ of each oscillator lifted to some interval of period  $2\pi$,  $\omega \in \mathbb{R}$ denotes the intrinsic frequency of the oscillator $x \in \mathbb{R}$, $F[\mu](x,\omega)$ is simultaneously $2\pi$-periodic in $\mu$ and $x$ (explained in \eqref{As3}) and  denotes the interaction of the whole population $\mu$ on a specific oscillator $x$ of frequency $\omega$. In that model the oscillators are distinguishable and their own frequencies $\omega$ are kept unchanged along the time as it can been observed in \eqref{A-1} that the derivative of the flux  with respect to $\omega$ is not present. In this sense, we can regard \eqref{A-1} as a ``{\it parameterized nonlocal continuity equation}" or equivalently a ``{\it parameterized scalar conservation law with a non-local flux}".  We refer to \cite{C-N-K-P, M-M-Z-2, M-M-Z-1, Z-M} for the related results on the nonlocal scalar conservation law. As time evolves the lifted distribution $\mu_t$ may span an interval larger than $2\pi$. That phenomenon appears for instance when two different uncoupled clusters evolved with distinct frequencies that may overlap mod $2\pi$ but not on $\mathbb{R}$. Such phenomenon could not have been observed if we had written the continuity equation in $\mathbb{S}^1$. Nevertheless, the interaction between the oscillators is short range (less than $2\pi$) and not long range as in $\mathbb{R}$.

In collective dynamics, finite-dimensional ODE models have been extensively studied in \cite{C-H-J-K}, and collective behaviors occur, when the interactions among particles are sufficiently strong. These collective dynamics is often measured by some nonlinear functionals. However, many previous studies have focused on the regime where the interaction is strong enough. In the case of a strong interaction the oscillators converge to a unique stationary configuration. On the contrary, a phase transition from an ordered state (coherence) to a disordered state (incoherence) is achieved when the coupling decreases. Yet, there are also many interesting phenomena in the intermediate interaction regime \cite{K-H-Y}. In \cite{O-K-T2, O-K-T}, in the case of a finite number of oscillators in the Winfree model, it is shown that if the initial data in space and in the natural frequencies are chosen in a suitable manner, then a time-periodic configuration with distinct oscillators is obtained. Indeed, the approach is extended to general finite-dimensional interacting particle systems in \cite{Oukil}. The motivation of our study is to generalize this phenomenon to an infinite-dimensional setting via the Cauchy problem \eqref{A-1} and a non discrete $\omega$-distribution. More precisely, we are interested in models where the $\omega$-distribution, that is the projection of $\mu_t$ onto the $\omega$-parameter, say $\nu(d\omega) = \mu_t(\mathbb{R},d\omega)$ (time independent), is any given distribution.

However, generalization of the analysis in finite-dimensional particle system to infinite-dimensional space of measures is not trivial. Using the method of characteristics, similar to the finite-dimensional case, we derive a Gr\"onwall-type differential inequality that governs the dispersion of the support of a probability measure. Nevertheless, a significant challenge lies in the application of Schauder-type fixed point theorem in infinite dimensional setting, as the dispersion relation does not directly give the compactness of Poincar\'e section in the invariant set. In other words, it is not easy to obtain both compactness and invariance along flow at the same time. 

In order to tackle this problem, we use disintegration of measure and optimal mass transport to define a proper metric in the space of measure-valued solutions to make it geodesically convex. Then, we obtain the Lipschitz property related to 2-Wasserstein distance by computing dispersion estimate of characteristic flow. This process gives rise to invariance set along the flow of the continuity equation. Inside this invariance set, we then construct a Poincar\'e section and use Schauder's fixed point theorem for the geodesically convex space to show the existence of a periodic measure-valued solution. 

Before we move on further, we recall the concept of a measure-valued solution to \eqref{A-1}. Let $\mathcal{P}(\mathbb{R}^2)$ denote the space of Borel probability measures on $\mathbb R^2$, and let $\mathcal P_2(\mathbb {R}^2)$ denotes the subset of those with the finite second moment:
\[
\int_{\mathbb{R}^2}(|x|^2+|\omega|^2) \mu(dx, d\omega) <\infty.
\]
For $\mu \in {\mathcal P}(\mathbb R^2)$ and a test function $\varphi \in {\mathcal C}_0(\mathbb R^2)$, we write
\[ \langle \mu, \varphi \rangle := \int_{\mathbb R^2} \varphi(x, \omega) \mu(dx, d\omega). \]
\begin{definition} \label{D1.1}
\emph{(Measure-valued solutions)}
Let $T \in (0, \infty)$. We say  that $\mu \in {\mathcal C}([0,T); \mathcal{P}_2(\mathbb{R}^2))$ is a measure-valued solution to \eqref{A-1} with initial datum $\mu_0 \in \mathcal{P}_2(\mathbb{R}^2)$  if the following conditions hold.
\begin{enumerate}
\item 
For all test function $\varphi \in {\mathcal C}_0(\mathbb R^2)$, 
\[ \mbox{the map}~t \mapsto~\langle \mu_t, \varphi \rangle~\mbox{is continuous}. \]
\item 
For all $\varphi \in {\mathcal C}^1_0([0, T) \times \mathbb{R}^2)$, $\mu$ satisfies the weak formulation of \eqref{A-1}:
\[
\langle \mu_t, \varphi(t)\rangle = \langle \mu_0, \varphi(0)\rangle+ \int_0^t \Big\langle \mu_s, \partial_s \varphi + F[\mu_s]\partial_x \varphi \Big\rangle ds.
\]
\end{enumerate}
\end{definition}
\vspace{0.2cm}
In this paper, we address the following issue on the existence of  a time-periodic graph measure solution: 
\vspace{0.1cm}
\begin{quote} 
``Is there a time-periodic measure-valued solution to \eqref{A-1} that models a collection of oscillators with a continuous distribution of intrinsic frequencies?"
\end{quote}
\vspace{0.2cm}
The main result of this paper is to show the existence of a time-periodic measure-valued solution to \eqref{A-1} as a fixed point of Poincar\'e map defined on a positively invariant set (see Theorem \ref{main}). The problem of the stability or the existence of many periodic solutions with an open basin of attraction is left to a future project.

The rest of this paper is organized as follows. In Section \ref{sec:2}, we first recall previous results for the special types of nonlocal continuity equation such as the kinetic Kuramoto and Winfree equations, and then we introduce several structural assumptions on the nonlocal continuity equation $\eqref{A-1}$. We also provide several preparatory lemmas for later sections. At the end of this section, we present our main results on the existence of a periodic measure-valued solution. In Section \ref{sec:3}, we introduce a metric using the disintegration of probability measures and Wasserstein distance, and then provide a sufficient condition for a metric space to be CAT(0) space. To best of our knowledge, it is the first time that some tools in hyperbolic geometry is used to exhibit particular solutions of the continuity equation. In Section \ref{sec:4}, we study a dispersion estimate using the characteristic flow to construct a positively invariant set along the dynamics generated by \eqref{A-1}. In Section \ref{sec:5}, as the subset of the invariance set constructed in previous section, we further construct an invariant set along the flow,  and then we use it to show the existence of a periodic measure-valued solution. In Section \ref{sec:6}, assuming a twist property for the vector field, we apply our method to prove the existence of a periodic graph measure solution which is a periodic solution with a further regularity. Finally, Section \ref{sec:7} is devoted to a brief summary of our main results and  some remaining issues for future works.

\section{Preliminaries}\label{sec:2}
\setcounter{equation}{0}
In this section, we study basic a priori estimates on the parameterized nonlocal continuity equation, and review related previous results for the kinetic Winfree and Kuramoto equations. 

\subsection{Previous results}
In this subsection, we briefly review the previous results for the particle Winfree model \cite{O-K-T2, O-K-T}:
\begin{equation}\label{E2.1}
\frac{dx_i}{dt} = \omega_i + \frac{\kappa}{N}\sum_{j =1}^{N} S(x_i) I(x_j), \quad i \in \{1,\dots,N\},
\end{equation}
where $S$ and $I$ are assumed to be $C^2$ and $2\pi$-periodic functions which are called as the sensitivity and influence functions, respectively. In literature \cite{A-S}, we often take the following ansatz:
\[ S(x) = -\sin x, \quad I(x):= 1 + \cos x.  \]
According to the choice of $S(x) = -\sin$, the phase $x=0$ is stable and the phase $x=-\pi$ is unstable. The Winfree model \eqref{E2.1} exhibits a phase transition from disordered state to ordered state, as the parameter $\kappa$ increases from zero. Namely, starting from the incoherent state, the configuration transits to partially locked-state and  completely locked state as the coupling strength increases. We refer to \cite{A-S,HPR} for a detailed discussion of the Winfree model and its phase-locking behavior. More precisely, it is shown in \cite{HPR} that for a sufficiently large coupling $\kappa \gg 1$, the Winfree model exhibits a phase-locked state. In contrast, there are very few works on the dynamics of the Winfree model in a small coupling regime. The authors in \cite{O-K-T2, O-K-T}, for a discrete distribution of frequencies, constructed an invariant set in a small coupling regime, and showed the existence of the periodic solution. In this work, we extend the results in \cite{O-K-T} to the corresponding mean-field equation \eqref{A-1} which corresponds to the infinite dimensional counterpart of the particle model. The crucial step to construct an invariant set is to obtain Gr\"onwall-type differential inequality and use a bootstrapping argument. 

We end this subsection by  discussing  the existence of a periodic solution for a very simple periodic vector field. The goal of this toy model is to explain the fundamental role of the assumption ``synchronization hypothesis'' \eqref{Equation:SynchronizationHypothesis} introduced later in section \ref{Subsection:SufficientFramework}. We consider the following  affine differential equation with periodic coefficients:
\begin{equation}
\begin{cases} \label{B-10}
\displaystyle {\dot y}(t) = \alpha+\beta(t) y(t), \quad t > 0,  \\
\displaystyle y \Big|_{t = t_0} = y_0,
\end{cases}
\end{equation}
where $\alpha$ is a positive constant, and $\beta$ is a ${\mathcal C}^1$ periodic function satisfying
\begin{equation} \label{B-11}
\beta(t + 2\pi) = \beta(t), \quad \forall~t \in \mathbb R, \quad \int_0^{2\pi}\beta(s)ds< 0.
\end{equation}
By the method of an integrating factor, one can find the explicit formula:
\begin{align}\label{B-11-0}
    y(t)=y_0\exp\left(\int_{t_0}^t\beta(s)ds\right)+\alpha\int_{t_0}^t\exp\left(\int_s^t\beta(\tau)d\tau\right)ds.
\end{align}
Note that the first term in the R.H.S. of \eqref{B-11-0} is not $2\pi$-periodic for a generic $y_0$:
\begin{align*}
    y_0\exp\left(\int_{t_0}^{t+2\pi}\beta(s)ds\right)=y_0\exp\left(\int_{t_0}^t\beta(s)ds+\int_0^{2\pi}\beta(s)ds\right)\ne y_0\exp\left(\int_{t_0}^t\beta(s)ds\right)
\end{align*}
unless $\beta(t)$ satisfies
\begin{align*}
    \int_0^{2\pi}\beta(s)ds=0.
\end{align*}
Unfortunately, our given condition $\eqref{B-11}_2$ does not satisfy the mean zero condition.  However, in the following lemma which extends the work in \cite{O-K-T}, we will see that for a special initial datum $y_0$, the Cauchy problem \eqref{B-10} yields the unique $2\pi$-periodic solution. 

\begin{lemma}\label{L2.2}
Let $y = y(t)$ be a global unique solution to \eqref{B-10} - \eqref{B-11} whose explicit representation is given by \eqref{B-11-0}. Then, the following assertions hold.
\begin{enumerate}
\item \label{Item:L2.2_1} The solution $y(t)$ is $2\pi$-periodic if and only if the initial datum $y_0$ is given by
\begin{align}\label{B-11-2}
    y_0=y(t_0)=\frac{\alpha\int_{t_0}^{t_0+2\pi}\exp\left(\int_s^{t_0}\beta(\tau)d\tau\right)ds}{\exp\left(-\int_0^{2\pi}\beta(s)ds\right)-1}.
\end{align}
In other words, the solution
\begin{align}\label{B-14-2}
    y(t)=\frac{\alpha\int_t^{t+2\pi}\exp\left(\int_s^t\beta(\tau)d\tau\right)ds}{\exp\left(-\int_0^{2\pi}\beta(s)ds\right)-1}
\end{align}
is the unique $2\pi$-periodic solution to $\eqref{B-10}_1$.

\item \label{Item:L2.2_2}
The periodic solution $y(t)$ satisfies the following upper and lower bound estimates:
\begin{align*}
    \frac{2\pi\alpha\exp\left(-\int_0^{2\pi}\beta^+(\tau)d\tau\right)}{\exp\left(-\int_0^{2\pi}\beta(s)ds\right)-1}\le y(t)\le\frac{2\pi\alpha\exp\left(\int_0^{2\pi}\beta^-(\tau)d\tau\right)}{\exp\left(-\int_0^{2\pi}\beta(s)ds\right)-1},
\end{align*}
where functions $\beta^{\pm}(s)$ are defined by
\[  \beta^-(s):= \max \{ 0,-\beta(s) \} \quad \mbox{and} \quad \beta^+(s) :=\max \{ 0,\beta(s) \}. \]

\item \label{Item:L2.2_3} 
Assume $\int_0^{2\pi} |\beta|(\tau) \,d \tau \leq \beta_{\max}$ and $\int_0^{2\pi} \beta(\tau) \, d\tau \leq - \beta_{min} < 0$. Then
\begin{gather}
\frac{\exp\left(\int_0^{2\pi}\beta^-(\tau)d\tau\right)}{\exp\left(-\int_0^{2\pi}\beta(s)ds\right)-1} \leq \frac{\exp( \beta_{max})}{\beta_{min}}. \label{Equation:L2.2_01}
\end{gather}

\end{enumerate}
\end{lemma}
\begin{proof} Although the verification of the assertions is straightforward, and it can be found in \cite{O-K-T}, we provide their proofs here for readers' convenience.

\vspace{.2cm}

\it Proof of Item \ref{Item:L2.2_1}. We split the proof into two cases.

\vspace{.2cm}

\noindent $\bullet$~Case A $(\Longrightarrow$ direction): Suppose that $y$ is $2\pi$-periodic, in particular, 
\begin{equation} \label{B-13}
y(t_0+2\pi) = y(t_0) = y_0. 
\end{equation}
Then, we use \eqref{B-11-0} and \eqref{B-13} to find 
\begin{align*}
    y_0\exp\left(\int_{t_0}^{t_0+2\pi}\beta(s)ds\right)+\alpha\int_{t_0}^{t_0+2\pi}\exp\left(\int_s^{t_0+2\pi}\beta(\tau)d\tau\right)ds=y_0,
\end{align*}
which satisfies \eqref{B-11-2}.

\vspace{0.2cm}

\noindent $\bullet$~Case B $(\Longleftarrow$ direction):  Suppose $y_0$ is given by the explicit value in \eqref{B-11-2}. Then, we use \eqref{B-11-0} to get
\begin{align}\label{B-14}
    \begin{aligned}
        y(t)&=\frac{\alpha\int_{t_0}^{t_0+2\pi}\exp\left(\int_s^{t_0+2\pi}\beta(\tau)d\tau\right)ds}{1-\exp\left(\int_0^{2\pi}\beta(s)ds\right)}\exp\left(\int_{t_0}^t\beta(s)ds\right)+\alpha\int_{t_0}^t\exp\left(\int_s^t\beta(\tau)d\tau\right)ds\\
        &=\frac{\alpha\int_{t_0}^{t_0+2\pi}\exp\left(-\int_{t_0}^s\beta(\tau)d\tau\right)ds}{\exp\left(-\int_0^{2\pi}\beta(s)ds\right)-1}\exp\left(\int_{t_0}^t\beta(s)ds\right)+\alpha\int_{t_0}^t\exp\left(\int_s^t\beta(\tau)d\tau\right)ds\\
        &=\frac{\alpha\int_{t_0}^{t_0+2\pi}\exp\left(\int_s^t\beta(\tau)d\tau\right)ds}{\exp\left(-\int_0^{2\pi}\beta(s)ds\right)-1}+\alpha\int_{t_0}^t\exp\left(\int_s^t\beta(\tau)d\tau\right)ds\\
        &=\frac{\alpha\int_{t_0}^t\exp\left(\int_s^t\beta(\tau)d\tau\right)ds+\alpha\int_{t+2\pi}^{t_0+2\pi}\exp\left(\int_s^t\beta(\tau)d\tau\right)ds}{\exp\left(-\int_0^{2\pi}\beta(s)ds\right)-1}\\
        &\hspace{.4cm}+\frac{\alpha\int_t^{t+2\pi}\exp\left(\int_s^t\beta(\tau)d\tau\right)ds}{\exp\left(-\int_0^{2\pi}\beta(s)ds\right)-1}+\alpha\int_{t_0}^t\exp\left(\int_s^t\beta(\tau)d\tau\right)ds.
    \end{aligned}
\end{align}
On the other hand, note that 
\begin{align}\label{B-14-1}
    \begin{aligned}
        \int_{t+2\pi}^{t_0+2\pi}\exp\left(\int_s^t\beta(\tau)d\tau\right)ds&=-\int_{t_0}^t\exp\left(\int_{s+2\pi}^t\beta(\tau)d\tau\right)ds\\
        &=-\exp\left(-\int_0^{2\pi}\beta(s)ds\right)\int_{t_0}^t\exp\left(\int_s^t\beta(\tau)d\tau\right)ds.
    \end{aligned}
\end{align}
Now, we combine \eqref{B-14} and \eqref{B-14-1} to see \eqref{B-14-2}. Finally, it is enough to show that $y(t)$ is $2\pi$-periodic. It follows from \eqref{B-14}  and \eqref{B-14-2} that 
\begin{align*}
    y(t+2\pi)&=\frac{\alpha\int_{t+2\pi}^{t+4\pi}\exp\left(\int_s^{t+2\pi}\beta(\tau)d\tau\right)ds}{\exp\left(-\int_0^{2\pi}\beta(s)ds\right)-1}=\frac{\alpha\int_{t}^{t+2\pi}\exp\left(\int_{\tilde s+2\pi}^{t+2\pi}\beta(\tau)d\tau\right)d\tilde s}{\exp\left(-\int_0^{2\pi}\beta(s)ds\right)-1}\\
    &=\frac{\alpha\int_{t}^{t+2\pi}\exp\left(\int_{\tilde s}^{t}\beta(\tau)d\tau\right)d\tilde s}{\exp\left(-\int_0^{2\pi}\beta(s)ds\right)-1}=y(t),
\end{align*}
where we used a change of variable $\tilde s = s - 2\pi$ and $2\pi$-periodicity of $\beta$.

\vspace{.2cm}

\it Proof of Item \ref{Item:L2.2_2}. We rewrite \eqref{B-14-2} as follows:
\begin{align}\label{B-15}
    \begin{aligned}
        y(t)&=\frac{\alpha\int_t^{t+2\pi}\exp\left(\int_s^t\beta(\tau)d\tau\right)ds}{1-\exp\left(\int_0^{2\pi}\beta(s)ds\right)}\exp\left(\int_0^{2\pi}\beta(s)ds\right)\\
        &=\frac{\alpha\int_0^{2\pi}\exp\left(-\int_t^{t+s}\beta(\tau)d\tau\right)ds}{\exp\left(-\int_0^{2\pi}\beta(s)ds\right)-1},
    \end{aligned}
\end{align}
where we used
\begin{align*}
    \begin{aligned}
        &\int_t^{t+2\pi}\exp\left(\int_s^t\beta(\tau)d\tau\right)ds\exp\left(\int_0^{2\pi}\beta(s)ds\right)\\
        &\hspace{1cm}=\int_t^{t+2\pi}\exp\left(\int_s^{t+2\pi}\beta(\tau)d\tau\right)ds=\int_0^{2\pi}\exp\left(\int_{t+\tilde s}^{t+2\pi}\beta(\tau)d\tau\right)d\tilde s\\
        &\hspace{1cm}=\int_0^{2\pi}\exp\left(\int_t^{t+2\pi}\beta(\tau)d\tau)-\int_t^{t+\tilde s}\beta(\tau)d\tau\right)d\tilde s\\
        &\hspace{1cm}=\int_0^{2\pi}\beta(\tau)d\tau\int_0^{2\pi}\exp\left(-\int_t^{t+\tilde s}\beta(\tau)d\tau\right)d\tilde s.
    \end{aligned}
\end{align*}
by using the change of variable $\tilde s=s-t$. Now, we take the decomposition of $\beta$ as follows:
\[ \beta = \beta^{+} - \beta^{-}\quad\mbox{where} \quad  \beta^+ \geq 0, \quad \beta^- \geq 0. \]
This deduce
\begin{align*}
    \begin{aligned}
        \exp\left(-\int_t^{t+s}\beta(\tau)d\tau\right)=\exp\left(-\int_t^{t+s}\beta^+(\tau)d\tau\right)\exp\left(\int_t^{t+s}\beta^-(\tau)d\tau\right),
    \end{aligned}
\end{align*}
hence, we get
\begin{align}\label{B-16}
    \begin{aligned}
        \exp\left(-\int_0^{2\pi}\beta^+(\tau)d\tau\right)\le\exp\left(-\int_t^{t+s}\beta(\tau)d\tau\right)\le\exp\left(\int_0^{2\pi}\beta^-(\tau)d\tau\right)
    \end{aligned}
\end{align}
due to $s\in(0,2\pi)$. Finally, we combine \eqref{B-15} and \eqref{B-16} to get the desired estimates. 

\it Proof of Item \ref{Item:L2.2_3}. This is a direct consequence of the inequality $\mathrm{exp}(x)-1\geq x $, applied to the denominator. 

\end{proof}

\subsection{A sufficient framework $\boldsymbol{({\mathcal A})}$} \label{Subsection:SufficientFramework}
Let $\mu\in\mathcal C([0,\infty);\mathcal P_2(\mathbb R^2))$ be a measure-valued solution to \eqref{A-1} and $\nu_t \in \mathcal{P}_2(\mathbb{R})$ be the $\omega$-marginal of $\mu_t$ for each time $t\ge0$. As for $\mathcal P_2(\mathbb R^2)$, the set $\mathcal{P}_2(\mathbb{R})$ denotes the set of probability measures on $\mathbb{R}$ with finite second moment. Since the equation \eqref{A-1} does not contain any derivative in the variable $\omega$, its marginal $\nu_t$ is independent of the time $t$. Hence, we consider the framework with respect to a fixed one-dimensional probability measure $\nu\in\mathcal P_2(\mathbb R)$ and denote its support by $\Omega$:
\[
\Supp(\nu):=\Omega\subset\mathbb R.\]
Throughout this paper, we assume that the domain $\Omega$ is bounded and define $\gamma$ to be the diameter of the support of $\nu$:
\begin{align}\label{nu_cond}
    |\Omega|<\infty\quad\mbox{and}\quad\gamma:=\sup_{\omega,\omega'\in\Omega}|\omega-\omega'|<\infty.
\end{align}
We denote by $\mathcal P_2^\nu(\mathbb R^2)$ the set of all probability measure $\mu\in\mathcal P_2(\mathbb R^2)$ with the $\omega$-marginal $\nu$, i.e.,
\begin{align*}
    \mathcal P_2^\nu(\mathbb R^2):=\Big\{\mu\in\mathcal P_2(\mathbb R^2) : \mu(\mathbb{R} \times B) =\nu(B),\ \forall\,  B\in\mathcal{B}(\mathbb{R}) \Big\},
\end{align*}
where $\mathcal{B}(\mathbb{R})$ denotes the set of  Borel sets of $\mathbb{R}$. 

As can be seen in Lemma \ref{L2.2}, even for a linear ODE with periodic coefficients, existence of periodic solution is a non-trivial matter. In what follows, we describe a set of conditions on the nonlocal velocity field $F[\mu]$ in \eqref{A-1} for the existence of a periodic solution. We point out that a positive parameter $\kappa$ represents the strength of the coupling between the oscillators that is supposed to be small, and assume that $ \gamma = \mathcal{O}(\kappa)$. The other parameters $A,B,I,M$ are not necessarily small and are supposed to represent the uncoupled model.

\begin{itemize}




\item
($\mathcal{A}_1$) (Mean field regularity hypothesis): For any $\mu, \mu' \in \mathcal P_2^\nu(\mathbb R^2)$, there exists a constant $M>0$ such that
\begin{equation*}\label{As1}
\sup_{(x,\omega)\in\mathbb R\times\Omega}\Big|\partial_x^\alpha F[\mu](x,\omega)-\partial_x^\alpha F[\mu'](x,\omega)\Big|\le\kappa MW_1(\mu,\mu'),\quad\forall~\alpha=0,1,
\end{equation*}
where $W_1$ denotes the 1-Wasserstein distance.

\vspace{0.1cm}

\item
($\mathcal{A}_2$) (Space regularity hypothesis): $F[\mu]$ is $C^2(\mathbb{R}^2)$  for any $\mu \in \mathcal P_2^\nu(\mathbb R^2)$. Moreover, there exists a constant $I>0$ such that 
\begin{gather*}
    \|\partial_x F[\mu]\|_{\infty}\leq \kappa I, \quad\|\partial^2_x F[\mu]\|_{\infty}    \leq \kappa I,\quad\|\partial_\omega F[\mu]\|_{\infty} \leq I.
\end{gather*}

\vspace{0.1cm}

\item
($\mathcal{A}_3$)(Lifted periodicity hypothesis):~$F[\cdot ](\cdot,\cdot)$ is $2\pi$-periodic in the following sense: 
\begin{equation}\label{As3}
F[\tau[2\pi] _\sharp \mu](x + 2\pi, \omega) = F[\mu](x, \omega),  \quad \forall\, \mu\in\mathcal P_2^\nu(\mathbb R^2),\  \forall\, (x,\omega) \in \mathbb{R}^2.
\end{equation}
where $\tau[2\pi] :\mathbb{R}^2 \rightarrow \mathbb{R}^2$ denotes the shift $\tau[2\pi](x, \omega) = (x +2\pi, \omega)$ and $\tau[2\pi]_\sharp \mu$ denotes the pushforward of $\mu$ by the shift. The physical meaning of \eqref{As3} states that if both an individual $x$ and its environment modeled by $\mu$ are translated by $2\pi$ then the coupling between $x$ and $\mu$ is unchanged.

\vspace{0.1cm}

\item
($\mathcal{A}_4$) (No stationarity  hypothesis): The vector field $F$ admits a uniform lower bound $A>0$. More precisely for a cluster of individuals located at $x$ and distribution of intrinsic frequencies $\nu$  the strength of the vector field satisfies
\begin{equation}\label{As4}
\inf_{x \in \mathbb{R}} F[\delta_x \otimes \nu](x, \omega_c) \geq A,
\end{equation}
where we define the  mean frequency $\omega_c$ by
\begin{align*}
    \omega_c:=\int_{\mathbb R}\omega \, \nu(d\omega).
\end{align*}

\vspace{0.1cm}

\item
($\mathcal{A}_5$) (Synchronization hypothesis):~There exists a constant $B>0$ such that
\begin{gather}
\int_0^{2\pi} \frac{\partial_x F[\delta_x \otimes \nu](x, \omega_c) }{F[\delta_x \otimes \nu](x, \omega_c)}dx\le-\kappa B<0. \label{Equation:SynchronizationHypothesis}
\end{gather}
The last hypothesis is similar to the hypothesis \eqref{B-11} for the toy model  and force the collective dynamics in the continuous setting  to be compactly supported.
\end{itemize}

\vspace{0.5cm}

For constants $A,B,I,M>0$ and parameters $\kappa,\gamma>0$, we set the collections $\mathcal F_{A,B,I,M}^\nu(\kappa,\gamma)$ of all functionals $F$ satisfying $(\mathcal A_1)-(\mathcal A_5)$, and omit the dependency on constants,
\[\mathcal F^\nu(\kappa,\gamma)=\mathcal F_{A,B,I,M}^\nu(\kappa,\gamma),\]
if the context is clear. Unless otherwise specified, we assume throughout this paper that the functional $F$ belongs to the class $\mathcal F^\nu_{A,B,I,M}(\kappa,\gamma)$ for suitable constants $A,B,I,M>0$ and parameters $\kappa,\gamma>0$. We will specify the conditions on these parameters later.

\begin{remark}
Let $F \in \mathcal F^\nu_{A,B,I,M}(\kappa,\gamma)$. 

\vspace{0.2cm}
{\it (1)} In ($\mathcal{A}_2$), we observe that the constant $\kappa$ is included only when the $x$-derivative is included. This is motivated by the property of the kinetic Kuramoto \eqref{kuramoto} and Winfree equations \eqref{winfree}. 

\vspace{0.2cm}
{\it (2)} We provide some motivation for $(\mathcal{A}_5)$. Consider a nonlinear continuous flow:
\begin{equation}\label{S1}
{\dot x} = F(x), \quad F(x + 2\pi) = F(x), \quad x \in \mathbb{R},
\end{equation}
and let $\varphi$ be the solution to \eqref{S1}: 
\begin{equation*} \label{C-0-0}
{\dot \varphi} = F(\varphi). 
\end{equation*}
We linearize the flow \eqref{S1} around the solution curve $\varphi(t)$. We set 
\[ x = \varphi + \Delta x \]
Then, the perturbation $\Delta x$ satisfies 
\begin{align*}
\frac{d}{dt}(\varphi + \Delta x) = \frac{d}{dt} \varphi + \frac{d}{dt}\Delta x = F(\varphi + \Delta x) = F(\varphi) + \partial_x F(\varphi) \Delta x + \mathcal{O}(1) \Delta x^2.
\end{align*}
Hence, the corresponding linear flow for the fluctuation is given as follows
\[
\frac{d}{dt}\Delta x = \partial_x F(\varphi) \Delta x.
\]
We do the change of variable $s := x(t)$. Then, we have
\[ \frac{ds}{dt} = {\dot x}(t) = F(x(t)) = F(s).
\] 
We now consider  $\Delta x$ as a function of $s$. Then, we have
\begin{equation*}
\frac{d}{ds}\Delta x = \frac{d}{dt} \Delta x \frac{dt}{ds}= \frac{\partial_x F(s)}{F(s)}\Delta x, \quad 
\Delta x(2\pi) = \Delta x(0) \exp\Big(  \int_0^{2\pi}  \frac{\partial_x F(s)}{F(s)} ds  \Big).
\end{equation*}
Thus, the condition similar to ($\mathcal{A}_5$)
\[
\int_0^{2\pi}\frac{\partial_x F(s)}{F(s)}ds < 0,
\]
implies 
\begin{gather*}
|\Delta x(2\pi)| =  |\Delta x(0)| \underbrace{\exp\Big(  \int_0^{2\pi}  \frac{\partial_x F(s)}{F(s)} ds  \Big)}_{\leq 1}, \\
|\Delta x(2n\pi)| =  |\Delta x(0)| \exp\Big( n \int_0^{2\pi}  \frac{\partial_x F(s)}{F(s)} ds  \Big) \quad \to \quad 0, \quad \mbox{as $n \to \infty$}. 
\end{gather*}

Therefore, we can see that $\Delta x$ tends to zero, as the argument tends to infinity. This is the reason why the condition $(\mathcal{A}_5)$ is also called the synchronization hypothesis.

\end{remark}

\subsection{Explicit examples for $\boldsymbol{\mathcal{F}^\nu(\kappa, \gamma)}$} \label{sec:2.3} In this subsection, we provide several simple examples in $\mathcal{F}^\nu(\kappa, \gamma)$ satisfying $(\mathcal{A}_1)$ - $(\mathcal{A}_5)$ from the collective dynamics.
\subsubsection{Kinetic Winfree flow} \label{sec:2.3.1}
Consider the nonlocal velocity field $F$:
\begin{equation}\label{winfree}
F[\mu](x, \omega) := \omega + \kappa \int_{\mathbb{R}^2} R(x) P(y) \mu(dy,d\omega).
\end{equation}
In literature \cite{A-S},  we often use the following explicit form:
\begin{equation*} \label{C-1}
R(x) = -\sin x \quad \mbox{and} \quad P(y) = 1 +\cos y. 
\end{equation*}
For $\mu = \delta_x \otimes \nu$, it is easy to see that 
\begin{align*}
F(\delta_x \otimes \nu,x,\omega_c) = \omega_c +  \kappa R(x) P(x), \quad  \partial_x F(\delta_x \otimes \nu,x,\omega_c) &=  \kappa R^{\prime}(x) P(x).
\end{align*}
This yields
\begin{align*}
\begin{aligned}
\frac{\partial_x F(\delta_x \otimes \nu,x,\omega_c)}{F(\delta_x \otimes \nu,x,\omega_c)} &= \frac{\kappa P(x)R'(x)}{\omega_c + \kappa P(x)R(x)}= \frac{\kappa(P'(x)R(x)+P(x)R'(x))}{\omega_c + \kappa P(x)R(x)} - \frac{\kappa P'(x)R(x)}{\omega_c + \kappa P(x)R(x)} \\
&=  \frac{(\omega_c + \kappa P(x)R(x))^{\prime}}{\omega_c + \kappa P(x)R(x)}  - \frac{\kappa P'(x)R(x)}{\omega_c + \kappa P(x)R(x)}.
\end{aligned}
\end{align*}

Thus, it follows from $(\mathcal{A}_5)$ that 
\[
Sync := \int_0^{2\pi} \frac{\partial_x F(\delta_x \otimes \nu ,x,\omega_c)}{F(\delta_x \otimes \nu ,x,\omega_c)} \, dx = -\kappa \int_0^{2\pi} \frac{ \sin^2 x}{\omega_c - \kappa (1+\cos x)\sin x} \, dx.
\]
Choosing as in \cite{O-K-T} a positive mean frequency $\omega_c >0$ and a coupling strength $\kappa$ small enough $\kappa \in (0,\frac{1}{4}\omega_c)$, $(\mathcal{A}_4)$ is satisfied since for every $x \in \mathbb{R}$
\begin{equation}\label{R2.1.1}
F[\delta_x \otimes \nu](x,\omega_c) = \omega_c + \kappa R(x)P(x) = \omega_c - \kappa \sin x (1 + \cos x) \geq A := \frac{1}{2} \omega_c >0.
\end{equation}
Furthermore $(\mathcal{A}_5)$ is also satisfied since
\begin{gather*}
Sync \leq  - \kappa B \ \ \text{with} \ \ B := \frac{1}{3}\omega_c >0.
\end{gather*}

\subsubsection{Kinetic Kuramoto flow} \label{sec:3.2.2} 
Consider the nonlinear velocity field $F$:
\begin{equation}\label{kuramoto}
F[\mu](x, \omega) := \omega - \kappa \int_{\mathbb{R}^2}\sin(x-y)  g(\omega) \mu(dy, d\omega),
\end{equation}
with $g\geq0$ and $\int g(\omega) \, \nu(d\omega) = 1$. For $\mu = \delta_x \otimes \nu$, we have
\begin{equation} \label{C-2}
F(\delta_x \otimes \nu,x,\omega_c) = \omega_c, \quad  \partial_x F(\delta_x \otimes \nu,x,\omega_c) =  -\kappa.
\end{equation}
This implies
\begin{align*}
\frac{\partial_x F(\delta_x \otimes \nu,x,\omega_c)}{F(\delta_x \otimes \nu,x,\omega_c)} &= -\frac{\kappa}{\omega_c} <0.
\end{align*}
Hence, we have
\begin{equation} \label{C-3}
\int_0^{2\pi} \frac{\partial_x F(\delta_x \otimes \nu,x,\omega_c)}{F(\delta_x \otimes \nu,x,\omega_c)} \, dx = -\kappa \frac{2\pi}{\omega_c}.
\end{equation}
It follows $\eqref{C-2}_1$ and \eqref{C-3} that the choices
\[ A =  \omega_c, \quad B= \frac{2\pi}{\omega_c}    \]
satisfy $(\mathcal{A}_4)$ and $(\mathcal{A}_5)$. 
%

\subsection{Collection of parameters and constants}
For readers' convenience, we collect and classify constants and parameters appearing throughout the paper.

\vspace{.2cm}

\begin{enumerate}
\item We fix positive \textit{constants} $A, B, I$ and $M$ appearing in the sufficient framework $(\mathcal{A})$. In practice, they are determined if the flux $F[\mu]$ is given specifically.
\item The positive numbers $\kappa, \gamma$ and $D$ are \textit{parameters} that will be chosen accordingly.
\item In the following, we define a set of constants depending on $A,B,I$ and $M$ as follows:
\begin{align}
\begin{aligned} \label{const}
& C_1(\kappa,\gamma,D) := \kappa( M+I)D^2 +2I\gamma, \\
& 
C_2(\kappa,\gamma,D) := \frac{C_1(\kappa,\gamma,D)}{A} +\frac{\left(((\kappa+1)I+ \kappa M)D+ I \gamma \right)\left(\kappa ID + C_1(\kappa,\gamma,D) \right)}{A\left(A - ((\kappa+1)I+ \kappa M)D - I \gamma\right)}, \\
&E_1(\kappa,\gamma,D):=\kappa\left(D\left(\frac{3I}{2}+M\right)+I\gamma\right)\\
&E_2(\kappa,\gamma,D)=\frac{\kappa}{A-((\kappa+1)I+ \kappa M)D-I\gamma}\left(D\left(\frac{3I}{2}+M\right)+I\gamma+\frac{I\big(\kappa(I+M)D+I\gamma\big)}{A}\right).
\end{aligned}
\end{align}
\end{enumerate}

\subsection{Preparatory lemmas}
For a given set of time-dependent measures $\{ \mu_t \}_{t\ge0}\subset\mathcal P_2(\mathbb R^2)$, we set the first spatial average:
\begin{equation} \label{D-0}
x_c[\mu_t] := \int_{\mathbb{R}^2} x \mu_t(dx, d\omega). 
\end{equation}
We write $x_{c}(t)$ instead of $x_c[\mu_t]$, if there is no need to clarify the measure $\mu_t$. The quantity $x_c(t)$ measures at time $t$ the averaged location of all the oscillators independently of their intrinsic frequency.

\begin{lemma}\label{L4.1}
Let $\mu  \in {\mathcal C}([0, \infty); \mathcal{P}_2(\mathbb{R}^2))$ be a global measure-valued solution to \eqref{A-1}. Then, the spatial average $x_c$ defined in \eqref{D-0} satisfies 
\begin{gather}\label{Equation:L4.1_01}
\frac{d}{dt}x_c(t) = \int_{\mathbb{R}^2}F[\mu_t](x,\omega) \mu_t(dx,d\omega), \quad  \forall\, t\ge0. 
\end{gather}
\end{lemma}
\begin{proof} We split the proof into three steps.

\vspace{.2cm}

\noindent $\bullet$~Step A: We claim that 
\begin{equation}\label{D-0-1}
x_c(t) = x_c(0) + \int_0^t \int_{\mathbb R^2}F[\mu_s]\mu_s(dx,d\omega)ds.
\end{equation}
For this, we choose a time-independent test function $\varphi \in \mathcal C^1_c(\mathbb{R}^2)$. Then, it follows from the weak formulation in Definition \ref{D1.1} and  $\partial_t \varphi = 0$ that 
\begin{align}\label{D-0-2}
\begin{aligned}
\int_{\mathbb{R}^2}\varphi \mu_t(dx,d\omega) = \int_{\mathbb{R}^2}\varphi \mu_0(dx,d\omega) + \int_0^t \int_{\mathbb{R}^2}\bigg(\partial_x \varphi F[\mu_s](x,\omega) \, \mu_s(dx,d\omega)\bigg)ds.
\end{aligned}
\end{align}
Now, we set
\[
\psi_N(x,\omega) := \psi(x) \mathcal{X}_N(x,\omega), \quad \text{for each} ~~ N\in \mathbb{N},
\]
where $\psi$ is any smooth sublinear function satisfying the growth relation:
\[ |\psi(x)| \leq ax +b \quad \mbox{for some $a>0$ and $b \in \mathbb{R}$}, \]
and $\mathcal{X}_N = \mathcal{X}_N(x,\omega)$ is a smooth cut-off function such that 
\begin{equation} \label{D-0-3}
\mathrm{supp}(\mathcal{X}_N) \subset \left [-N-\frac{1}{2},N+\frac{1}{2} \right ] \times \mathbb{R} \quad \mbox{and} \quad \mathcal{X}_N = 1 \quad \mbox{if $x \in [-N,N]$}. 
\end{equation}
Note that  the integral $\displaystyle \int_{\mathbb{R}^2} \psi_N \mu_t(dx, d\omega)$ is well-defined due to the fact that 
 $\mu \in {\mathcal C}([0, \infty);\mathcal{P}_2(\mathbb{R}^2))$.
Since $\psi_N$ is a test function, it follows from \eqref{D-0-2} that
\begin{align*}
\int_{\mathbb{R}^{2}} \psi_N \mu_t(dx, d\omega) = \int_{\mathbb{R}^{2}} \psi_N \mu_0(dx, d\omega) +  \int_0^t \int_{\mathbb{R}^{2}}\partial_x \psi_N F[\mu_s](x,\omega)\mu_s(dx,d\omega)ds.
\end{align*}
Letting $N \rightarrow \infty$,  we use \eqref{D-0-3} to see that for each $(x,\omega) \in \mathbb R^2$, 
\begin{align}
\begin{aligned} \label{D-0-4}
& \psi_N(x, \omega) \to \psi(x,\omega) \quad \mbox{and} \\
& \partial_x \psi_N(x,\omega) = \partial_x \psi \mathcal{X}_N(x,\omega) + \psi \partial_x \mathcal{X}_N(x,\omega) \to \partial_x \psi(x, \omega).
\end{aligned}
\end{align}
Therefore, we use \eqref{D-0-4} and the Lebesgue dominated convergence theorem to obtain the relation:
\begin{align} \label{D-0-5}
\int_{\mathbb{R}^2} \psi\mu_t(dx, d\omega) = \int_{\mathbb{R}^2} \psi \mu_0(dx, d\omega) + \int_0^t \int_{\mathbb{R}^2}\partial_x \psi F[\mu_s](x,\omega)\mu_s(dx,d\omega)ds.
\end{align}
Now, we set $\psi = x$ and use \eqref{D-0-5} to derive the claim \eqref{D-0-1}. 

\vspace{0.2cm}

\noindent $\bullet$~Step B: ~Here, we show that $x_c(t)$ is differentiable in time and satisfies \eqref{Equation:L4.1_01}. It is sufficient to prove that $s \mapsto \int_{\mathbb{R}^2}  F[\mu_s](x,\omega) \, \mu_s(dx,d\omega)$ is continuous with respect to the Wasserstein distance. Indeed
\begin{align*}
\Big| \int_{\mathbb{R}^2} F[\mu_s](x,\omega) \, \mu_s(dx,d\omega) &- \int_{\mathbb{R}^2}  F[\mu_t](x,\omega) \, \mu_s (dx,d\omega) \Big| \\
&\leq \Big| \int_{\mathbb{R}^2}  F[\mu_s](x,\omega) \, \mu_s(dx,d\omega) - \int_{\mathbb{R}^2}  F[\mu_t](x,\omega) \, \mu_s (dx,d\omega) \Big| \\
&\quad+ \Big| \int_{\mathbb{R}^2}  F[\mu_t](x,\omega) \, \mu_s(dx,d\omega) - \int_{\mathbb{R}^2}  F[\mu_t](x,\omega) \, \mu_t (dx,d\omega) \Big| \\
&\leq (\kappa M + (\kappa+1)I) W_1(\mu_s,\mu_t).
\end{align*}
The first bound $\kappa M$ comes from the Lipschitz regularity of $s \mapsto \mu_s$ given in  ($\mathcal{A}_1$), the second bound $(\kappa+1)I$ comes from the Lipschitz regularity of $F[\mu_t]$ given in ($\mathcal{A}_2$).

\end{proof}
Next, we provide the estimates for $x_c(t)$ via approximate dynamics. We consider a fictitious distribution $\tilde \mu$ in $\mathbb{R}^2$ as if all the oscillators were located at $x_c$.
\begin{lemma}\label{L4.2}
Let $\mu \in {\mathcal C}([0,\infty); \mathcal{P}_2(\mathbb{R}^2))$ be a global measure-valued solution to \eqref{A-1} and let  $\{ \tilde{\mu}_t \}_{t\ge0}$ be a time-dependent family of measures defined by 
\begin{gather} 
\tilde{\mu}_t : = \delta_{x_c(t)} \otimes \nu,\quad\forall~t\ge0. \label{Equation:LeaderDistribution}
\end{gather}
Then, we have 
\begin{equation*} \label{D-1}
\bigg| \frac{d}{dt}x_c(t) - F[\tilde{\mu}_t](x_c(t),\omega_c)\bigg| \leq ((\kappa+1)I+\kappa M)W_{1}(\mu_t, \tilde{\mu_t})+  I\gamma.
\end{equation*}
\end{lemma}
\begin{proof}
By Lemma \ref{L4.1}, it amounts to compare two quantities
\[
\int_{\mathbb{R}^2}F[\mu_t] \mu_t(dx,d\omega) \quad \text{and}\quad  F[\tilde{\mu}_t](x_c(t),\omega_c).
\]
We see that
\begin{align}
\begin{aligned} \label{D-2}
&\int_{\mathbb{R}^2}F[\mu_t](x,\omega)\mu_t(dx,d\omega)-  F[\tilde{\mu}_t](x_c(t),\omega_c)\\
& \hspace{0.5cm} = \int_{\mathbb{R}^2} F[\mu_t](x,\omega)\mu_t(dx,d\omega)-\int_{\mathbb{R}^2}F[\tilde{\mu}_t](x,\omega)\tilde{\mu}_t(dx,d\omega)\\
& \hspace{1cm}+\int_{\mathbb{R}^2}F[\tilde{\mu}_t](x,\omega)\tilde{\mu}_t(dx,d\omega)-  F[\tilde{\mu}_t](x_c(t),\omega_c)\\
& \hspace{0.5cm} = \int_{\mathbb{R}^2}F[\mu_t](x,\omega)(\mu_t -\tilde{\mu}_t)(dx,d\omega) + \int_{\mathbb{R}^2}\bigg(F[\mu_t](x,\omega) - F[\tilde{\mu}_t](x,\omega)\bigg)\tilde{\mu}_t(dx,d\omega)\\
&\hspace{1cm}+\int_{\mathbb{R}^2}F[\tilde{\mu}_t](x,\omega)\tilde{\mu}_t(dx,d\omega)-  F[\tilde{\mu}_t](x_c(t),\omega_c)\\
& \hspace{0.5cm} =: \mathcal{I}_{11} + \mathcal{I}_{12} + \mathcal{I}_{13}.
\end{aligned}
\end{align}
In the sequel, we estimate the above terms ${\mathcal I}_{1i}$ one by one. 

\vspace{.2cm}

\noindent $\bullet$~Case A.1 (Estimate of ${\mathcal I}_{11}$):  By $(\mathcal{A}_2)$ and Kantorovich-Rubinstein duality formula, we have
\begin{align}\label{T-1}
    |\mathcal I_{11}|\le\text{Lip}(F[\mu])W_1(\mu_t,\tilde\mu_t)\le(\kappa+1)IW_1(\mu_t,\tilde\mu_t).
\end{align}
\noindent $\bullet$~Case A.2 (Estimate of ${\mathcal I}_{12}$):  We use $(\mathcal{A}_1)$ to get 
\[
\big|F[\mu_t](x, \omega) - F[\tilde{\mu}_t](x, \omega)\big| \leq \kappa MW_{1}(\mu_t, \tilde{\mu_t}).
\]
This yields
\begin{equation}\label{T-2}
\mathcal{I}_{12} \leq \kappa MW_{1}(\mu_t, \tilde{\mu_t}).
\end{equation}

\noindent $\bullet$~Case A.3 (Estimate of ${\mathcal I}_{13}$):~We use the Lipschitz estimate of $F[\tilde\mu_t]$ with respect to $\omega$ given in  $(\mathcal{A}_2)$ to obtain 
\begin{align}\label{T-3}
\mathcal{I}_{13} &\leq \Big| \int_{\Omega} F[\tilde \mu_t](x_c(t),\omega) -F[\tilde\mu_t](x_c(t), \omega_c) \, \nu(d\omega) \Big|, \notag \\
&\leq \|\partial_\omega F\|_\infty \cdot \gamma \leq I\gamma.
\end{align}

In \eqref{D-2}, we collect all the estimates \eqref{T-1}, \eqref{T-2} and \eqref{T-3} to obtain the desired estimate:
\begin{equation*}
\int_{\mathbb{R}^2}F(\mu_t,x,\omega)\mu_t(dx,d\omega)-  F(\tilde{\mu}_t,x_c(t),\omega_c) \leq ((\kappa+1)I+\kappa M)W_{1}(\mu_t, \tilde{\mu_t}) +I \gamma.
\end{equation*}
\end{proof}
\subsection{Description of main results} \label{sec:2.6} In this subsection, we state our main result whose proof will be provided  later (see Section \ref{sec:5}). First, we recall below the disintegration theorem.

\begin{theorem}[Disintegration theorem]\cite{A-G-S}\label{thm_dis}
    Let $X,Y$ be Radon separable metric spaces, $\mu\in\mathcal P(X)$, let $\pi:X\to Y$ be a Borel measurable map and let $\nu=\pi_\sharp\mu\in\mathcal P(Y)$. Then there exists a $\nu$-a.e. uniquely determined family of probability measures $\{\mu(\cdot,y)\}_{y\in Y}\subset\mathcal P(X)$, called conditional measures, satisfying the following three properties:
    \begin{enumerate}
        \item The map $y\mapsto \mu(\cdot, y)$ is measurable, i.e., $y\mapsto\mu(A,y)$ is measurable for every Borel measurable set $A$ of $X$.
        \item For $\nu$-almost all $y\in Y$,
        \begin{align*}
            \mu(X\setminus\pi^{-1}(y),y)=0.
        \end{align*}
        \item For every Borel map $f:X\to[0,\infty]$,
        \begin{align*}
        \int_Xf(x) \, \mu(dx)=\int_Y\left(\int_{\pi^{-1}(y)}f(x) \, \mu(dx,y)\right) \, \nu(dy).
    \end{align*}
    \end{enumerate}
\end{theorem}

The existence of a periodic solution in the space of probabilities in $\mathbb{R}^2$  will be obtained as a fixed point of some Poincar\'e map restricted to a closed convex set using Schauder fixed point theorem.  However, as noted in Section \ref{sec:1}, choosing an appropriate topology is a delicate problem. For a discrete density of frequencies $\frac{1}{N} \sum_{i=1}^N \delta_{\omega_i}$ and a discrete initial condition $\mu_0=\frac{1}{N}\sum_{i=1}^N \delta_{(X_i^0,\omega_i)}$, the solution of \eqref{A-1} is given by  $\mu_t = \frac{1}{N}\sum_{i=1} \delta_{(X_i(t),\omega_i)}$ where $X(t):=(X_i(t))_{i=1}^N$. In that case, the oscillators $X(t)$ are distinguishable, parametrized by the frequency index $i \in [N]:=\{1,\cdots,N\}$. The notion of convexity between two such probabilities $\mu_X =  \frac{1}{N}\sum_{i=1} \delta_{(X_i,\omega_i)}$ and $\mu_Y =  \frac{1}{N}\sum_{i=1} \delta_{(Y_i,\omega_i)}$ is easily defined by taking the barycenter of the corresponding points $(1-\alpha) \mu_X \oplus \alpha \mu_Y := \frac{1}{N} \sum_{i=1}^N \delta_{((1-\alpha)X_i+\alpha Y_i,\omega)}$. For a non discrete distribution $\nu$ we use another topology given by $L^2(\Omega,\nu;(\mathcal{P}_2^\nu(\mathbb{R}^2),W_2))$ that gives us an Hilbert-like structure where $W_2$ is the 2-Wasserstein distance. We will apply the following extension of Schauder fixed point theorem to establish the existence of a fixed point.

\begin{theorem}\label{Sch_thm}\cite{A-L-L, N-R} 
Let  $(X,d)$ be a $CAT(0)$ geodesic space and $K \subseteq X$ be a closed geodesically convex subset. If $T:K \to K$ is a continuous map such that $\overline{T(K)}$ is compact, then $T$ admits a fixed point in $K$.
\end{theorem}
To state our main result, we set the diameter in $x$ of $\mu$ as follows:
\begin{align*}
    \diam_x(\Supp(\mu)):=\sup\{|x-x'|~:~(x,\omega),(x',\omega')\in\Supp(\mu)\}.
\end{align*}

Furthermore, we define the set $\mathcal C_\nu(\Delta,\tilde\Delta)$ as the collection of probability measures $\mu\in\mathcal P_2^{\nu}(\mathbb R^2)$ admitting conditional measures in $\omega$, Lipschitz for the 2-Wasserstein distance. More precisely  
\begin{align}\label{C_nu_def}
\begin{aligned}
\mathcal{C}_\nu(\Delta,\tilde \Delta) :=& \{ \mu\in\mathcal P_2^{\nu}(\mathbb R^2) : \\
    &\quad\quad(i)~\diam_x(\Supp(\mu))\le\Delta(x_c[\mu]),\\
    &\quad\quad(ii)~W_2(\mu(\cdot,\omega),\mu(\cdot,\omega'))\le\tilde\Delta(x_c[\mu])|\omega-\omega'|,\quad\forall~\omega,\omega'\in\Omega \big\},
\end{aligned}
\end{align}
where the functions $\Delta,\tilde\Delta:\mathbb R\to\mathbb R^+$ are positive $2\pi$-periodic functions, which will be specified in Proposition \ref{P4.1} and Proposition \ref{P5.1}, respectively. Notice that we restrict our study of \eqref{A-1} for initial conditions $\mu_0$ satisfying \eqref{C_nu_def} that is possessing Lipschitz conditional disintegrations with respect to the frequencies.
\begin{theorem}[Main result]\label{main}
Let $\nu \in\mathcal{P}_2(\mathbb{R})$, $F \in \mathcal F^\nu_{A,B,I,M}(\kappa,\gamma)$, and $\kappa,\gamma$ chosen as in  \eqref{param1},  \eqref{D_cond}, and \eqref{Equation:LipschitzRegularityAssumption}. Let $\Delta$ and $\tilde\Delta$ be the $2\pi$-periodic functions given in Propositions \ref{P4.1} and \ref{P5.1}. The following assertions hold.
    \begin{enumerate}
        \item The subset $\mathcal C_\nu(\Delta,\tilde\Delta)$ of $\mathcal P_2^\nu(\mathbb R^2)$ is positively invariant under the flow:  if $\mu_0\in\mathcal C_\nu(\Delta,\tilde\Delta)$ and $\mu\in\mathcal C([0,\infty);\mathcal P_2(\mathbb R^2))$ is a measure-valued solution of \eqref{A-1} with initial datum $\mu_0$ then $\mu \in \mathcal{C}_\nu$ for all $t\geq0$.
        \item There exists $\mu_*\in\mathcal C_\nu(\Delta,\tilde\Delta)$ such that the solution $\mu$ of \eqref{A-1} with  initial condition $\mu_0=\mu_*$ is periodic in the following sense:
        \begin{align*}
            \mu_{t+T_*}=\tau[2\pi]_\sharp\mu_t,\quad\forall~t\ge0,
        \end{align*}
        where the translation map $\tau[2\pi]:\mathbb R\to\mathbb R$ is defined by $\tau[2\pi](x):=x+2\pi$ and \[T_*:=\inf\{t\ge0:x_c[\mu_t]=x_c[\mu_0]+2\pi\}.\]
    \end{enumerate}
\end{theorem}
See Remarks \ref{Remark:Feasibility} and  \ref{Remark:FeasibilityBis} for the feasibility of the three assumptions \eqref{param1}, \eqref{D_cond}, and \eqref{Equation:LipschitzRegularityAssumption} for a set of parameters of the form $\{ (\kappa,\gamma) : \kappa \in (0,1), \ \gamma = \mathcal{O}(\kappa) \}$, $\Delta  = \mathcal{O}(1)$, and $\tilde\Delta = \mathcal{O}(\kappa^{-1})$.


\section{Topology of the metric space}\label{sec:3}
\setcounter{equation}{0}
In this section, we analyze the topology of a metric space of measures, and construct a geodesically convex subspace. This construction enables us to apply the Schauder fixed point theorem to obtain the fixed point of the Poincar\'e map. The correspondence between the involved spaces, which is crucial for applying our framework, will be clarified later in the proof of Theorem \ref{main} (see Subsection \ref{sec:5.2}).

\subsection{The metric space $\boldsymbol{(\mathcal{X}_\nu, d_2)}$}
We recall that a Markov kernel from a measurable space \((\Omega, \mathcal{F})\) to another measurable space \((\mathbb R, \mathcal{B}(\mathbb R))\) is a map
\[
p: \mathcal{B}(\mathbb R)\times\Omega \to [0,1]\]
such that for each \(\omega \in \Omega\), the map \(A \mapsto p(A,\omega)\) is a probability measure on \(\mathbb R\), and for each \(A \in \mathcal{B}(\mathbb R)\), the map \(\omega \mapsto p(A,\omega)\) is \(\mathcal{F}\)-measurable. A Markov kernels $p(dx,\omega)$ is sometimes written $p(dx|\omega)$ to indicate a probability conditional to the information $\omega$.

We now introduce a metric space of Markov kernels that will be used in the application of the Schauder fixed point theorem. 
\begin{definition}
    Let $\nu$ be a compactly supported probability measure on $\Omega$. We define $\mathcal X_\nu$ to be the set of Markov kernels $p:\mathcal B(\mathbb R)\times\Omega\to[0,1]$ satisfying
    \[\int_{\mathbb{R}} |x|^2 \, p(dx,\omega) < \infty.\]
    We equip $\mathcal X_\nu$ with the metric
    \begin{align}\label{distance}
        d_2(p, \tilde p) := \left( \int_\Omega W_2^2(p(\cdot,\omega), \tilde p(\cdot,\omega)) \, \nu(d\omega) \right)^{1/2},
    \end{align}
    where \(W_2\) denotes the 2-Wasserstein distance on \(\mathbb{R}\). Moreover, we identify kernels \(p\) and \(\tilde p\) whenever \(d_2(p, \tilde p) = 0\), i.e., if they coincide \(\nu\)-almost everywhere.
\end{definition}
Let $\pi:\mathbb R\times\Omega\to\Omega$ denote the projection onto the second argument, i.e., $\pi(x,\omega)=\omega$. Note that the disintegration theorem (Theorem \ref{thm_dis}) implies that the set $\mathcal X_\nu$ is identified with the space
\begin{align*}
    \mathcal P_2^{\nu}(\mathbb R^2):=\left\{\mu\in\mathcal P_2(\mathbb R^2)~:~\pi_\sharp\mu=\nu\right\}.
\end{align*}
Moreover, since the dynamics given by \eqref{A-1} preserve the \(\omega\)-marginal, the pushforward \(\nu=\pi_\sharp \mu_t\) is independent of the time \(t\). Again applying the disintegration theorem with respect to \(\pi\), each \(\mu_t\) can be written as
\[
\mu_t(dx, d\omega) = K_t(dx, \omega)\, \nu(d\omega),
\]
where \(K_t :  \mathcal{B}(\mathbb R)\times\Omega \to [0,1]\) is the corresponding Markov kernel. Furthermore, under the assumption that \(\mu_t\) has finite second moment in \(x\), the kernels \(K_t(dx, \omega)\) also have finite second moments for \(\nu\)-almost every \(\omega\) and all \(t\). Thus, we identify the solution \(\mu_t\) to \eqref{A-1} with a time-dependent family of Markov kernels \(K_t(dx,\omega)\), where each \(K_t \in \mathcal{X}_\nu\).

As $\omega \in \Omega$ is stationary along the flow, the dynamics of a probability measure $\mu$ is essentially one-dimensional. From the optimal transport theory see for instance \cite{Santam}, the probability measures in one-dimensional space have a special property. Namely, the Wasserstein distances can be computed using the pseudo-inverse cumulative distribution functions.

Note that a  Markov kernel $p \in \mathcal{X}_\nu$ defines a conditional cumulative distribution function $F_p$ and a conditional pseudo-inverse $F^{-1}_p$. Specifically, for every  $s \in (0,1)$ and $\omega \in \Omega$, 
\[
F_p(x,\omega) := p((-\infty,x], \omega) \quad \mbox{and} \quad  F^{-1}_p(s,\omega) := \inf \{x \in \mathbb{R} : F_p(x,\omega) \geq s \}.
\]
Conversely a Markov kernel $p \in \mathcal{X}_\nu$ can be recovered  from $F^{-1}_p$:
\[
p(\cdot, \omega) = F^{-1}_p(\cdot,\omega)_\sharp \mathcal{L}, \quad \text{for all} \quad \omega \in \Omega, 
\]
where $\mathcal{L}$ is the Lebesgue measure restricted on $(0,1)$. A path from $p_0\in \mathcal{X}_\nu$ to $p_1\in \mathcal{X}_\nu$ is the map $u \in [0,1] \mapsto (1-u)p_0 \oplus u p_1 \in \mathcal{X}_\nu$ defined by: 
\[
((1-u)p_0 \oplus u p_1) (\cdot, \omega) := ((1-u)F^{-1}_{p_0}(\cdot,\omega) + u F^{-1}_{p_1}(\cdot,\omega))_\sharp \mathcal{L}, \quad \forall~\omega \in \Omega.
\]
Moreover, it follows from the standard theory of one-dimensional optimal transport \cite{Santam} that 
\[
d_2(p_0,p_1) = \left(\int_\Omega \int_0^1 \Big |F^{-1}_{p_0}(s,\omega)-F^{-1}_{p_1}(s,\omega) \Big |^2 ds \, \nu(d\omega)\right)^{1/2}.
\]
In other words, the distance \eqref{distance} can be computed using the inverse cumulative distribution function. 

\subsection{Geodesic property}  \label{sec:3.2}
In the space $\mathcal{P}_2(\mathbb{R})$, the path $u \in [0,1] \mapsto (1-u)p_0 \oplus u p_1 \in \mathcal{X}_\nu$ is indeed a geodesic path between $p_0$ and $p_1$. In particular, the metric space $(\mathcal{X}_\nu, d_2)$ becomes a CAT(0) space.  

\begin{lemma} \label{Lemma:CATspace}
The metric space $(\mathcal{X}_\nu,d_2)$ is an exact CAT(0) uniquely geodesic space:
\begin{enumerate}
\item $(\mathcal{X}_\nu,d_2)$ is uniquely geodesic: for every $p_0,p_1 \in \mathcal{X}_\nu$ the path $u \mapsto  p_u := (1-u)p_0 \oplus up_1$ is the unique geodesic connecting $p_0$ and $p_1$,
\begin{gather*}
d_2(p_u,p_v)= |u-v| d_2(p_0,p_1), \quad \forall~ u,v \in [0,1].
\end{gather*}
\item $(\mathcal{X}_\nu,d_2)$ is a CAT(0) space: for every $p,p_0,p_1 \in \mathcal{X}_\nu$ and $u \in [0,1]$,
\begin{gather}
d_2(p,p_u)^2 \leq (1-u) \, d_2(p,p_0)^2 + u \, d_2(p,p_1)^2 -u(1-u) \, d_2(p_0,p_1)^2. \label{Equation:CATspace_01}
\end{gather}
\end{enumerate}

\end{lemma}

\begin{proof}
Recall that the 2-Wasserstein space $\mathcal{P}_2(\mathbb{R})$ is an exact CAT(0) geodesic space (see \cite{Klo} or  \cite[Chapter 2]{Santam} for further details.) First, for every $\tilde p_0, \tilde p_1 \in \mathcal{P}_2(\mathbb{R})$, the path $u \mapsto \tilde p_u = (1-u) \tilde p_0 \oplus u  \tilde p_1$ defined by their inverse cumulative distribution functions $F^{-1}_{\tilde p_0}$ and $F^{-1}_{\tilde p_1}$ satisfies the following relation:~for every $u,v \in [0,1]$
\begin{gather*}
W_2(\tilde p_u , \tilde p_v)^2 = |u-v|^2\int_0^1 |F^{-1}_{\tilde p_0}(s)-F^{-1}_{\tilde p_1}(s)|^2 ds = |u-v|^2 W_2(\tilde p_0,\tilde p_1)^2.
\end{gather*}
Thus, the path $u \mapsto \tilde{p}_u$ is a geodesic path between $\tilde p_0$ and $\tilde p_1$.
From this, for every $\tilde p \in \mathcal{P}_2(\mathbb{R})$ we have
\begin{gather*}
W_2(\tilde p , \tilde p_u)^2 = (1-u) W_2(\tilde p , \tilde p_0)^2 + u W_2(\tilde p , \tilde p_1)^2 -u(1-u) W_2(\tilde p_0 , \tilde p_1)^2.
\end{gather*}
Now, the proof follows by setting $\tilde p_0 = p_0(\cdot , \omega)$ and $\tilde p_1 = p_1(\cdot, \omega)$ and by the integration with respect to the measure $\nu$.
\end{proof}

\begin{remark}
The inequality \eqref{Equation:CATspace_01} in Lemma \ref{Lemma:CATspace} is actually an equality:  we say that $(\mathcal{X}_\nu,d_2)$ is an exact CAT(0) space. However, we will not use that fact in this work.
\end{remark}

\subsection{Geodesic convexity and compactness} \label{sec:3.3}
In this subsection, we construct a subspace of the metric space $(\mathcal{X}_\nu, d_2)$ which is geodesically convex and compact. We restrict ourselves to Markov kernel $p(\cdot,\omega)$ that are Lipschitz with respect to $\omega$. 

For given parameters $\mathcal D,L >0$, we define the subspace $\mathcal X'_\nu(\mathcal D,L)$ of Markov kernels $p \in \mathcal{X}_\nu$ satisfying the following three properties:
\begin{enumerate}
	\item The center of mass on the $x$-space is located at zero,
	\begin{equation} \label{New-C-1}
		\int_\Omega \int_\mathbb{R} x p(dx, \omega)\nu(d\omega)=0.
	\end{equation}
	\item For any $\omega,\omega' \in \Supp(\nu), x \in \Supp(p(\cdot,\omega))$ and $x' \in \Supp(p(\cdot,\omega')),$ we have 
	\begin{equation} \label{New-C-2}
		|x-x'| \leq \mathcal D.
	\end{equation}
	\item The Kernel $p$ is Lipschitz in $\omega$ with respect to the 
	2-Wasserstein distance, i.e.,
	\begin{equation*} \label{New-C-3}
		W_2(p(\cdot, \omega),p(\cdot,\omega')) \leq L |\omega - \omega'|,\quad\forall~\omega,\omega'\in\Omega.
	\end{equation*}
\end{enumerate}
In the following lemma, we study the properties of $\mathcal X'_\nu(\mathcal D,L)$.

\begin{lemma} \label{Lemma:CAT0CompactConvex}
	The space $\mathcal X'_\nu(\mathcal D,L)$ satisfies the following statements:
	\begin{enumerate}
		\item \label{Item:CAT0CompactConvex_01} For any $p \in \mathcal X'_\nu(\mathcal D,L),  \omega \in \Supp(\nu)$ and $x \in \Supp(p(\cdot,\omega))$, we have 
		\[ |x| \leq \mathcal D. \]
		Thus the support of $p(\cdot, \omega)$ is contained in the interval $[-\mathcal D, \mathcal D]$.
		\vspace{0.2cm}
		\item \label{Item:CAT0CompactConvex_02}   $\mathcal X'_\nu(\mathcal D,L)$ is a  geodesically convex subset of $(\mathcal{X}_\nu,d_2)$.
		\vspace{0.2cm}
		\item \label{Item:CAT0CompactConvex_03}  $\mathcal X'_\nu(\mathcal D,L)$ is a compact subset of $(\mathcal{X}_\nu,d_2)$.
	\end{enumerate}
\end{lemma}
\begin{proof} In the sequel, we verify each statement one-by-one. \newline
	
	\noindent {\it Proof of item (\ref{Item:CAT0CompactConvex_01}).} By \eqref{New-C-2}, for each $\omega' \in \Supp(\nu)$ and $x' \in \Supp(p(\cdot,\omega'))$, we have 
	\[ |x-x'|\leq \mathcal D. \]
	Then, we use \eqref{New-C-1} to obtain 
	\[
	|x| = \bigg|x - \int_{\Omega}\int_{\mathbb R} x' p(dx',\omega') \nu(d\omega') \bigg| \leq \int_{\Omega}\int_{\mathbb R} |x- x'| \,  p(dx',\omega') \nu(d\omega') \leq \mathcal D.
	\]
	
	\noindent {\it Proof of item (\ref{Item:CAT0CompactConvex_01}).}  Let $p_0$ and $p_1$ be two points of $\mathcal X'_\nu(\mathcal D,L)$, and $F^{-1}_0$ and $F^{-1}_1$ be pseudo-inverses, respectively. For some $u \in [0,1]$, we define 
	\[ p := (1-u) p_0 \oplus up_1. \]
	Then, we have
	\[  p(\cdot, \omega) = F^{-1}(\cdot, \omega)_\sharp \mathcal{L}, \quad  F^{-1} := (1-u)F^{-1}_0 + u F^{-1}_1. \]
	Now, it suffices to show
	\[ p \in \mathcal X'_\nu(\mathcal D,L). \] 
	Using the pseudo-inverse, the first condition is immediately satisfied:
	\begin{align*}
		\begin{aligned}
			&\int_\Omega \Big( \int_\mathbb{R} x p(dx, \omega) \Big) \nu(d\omega) = \int_\Omega \Big( \int_0^1 F^{-1}(s,\omega)ds \Big) d\nu(\omega) \\
			& \hspace{1cm} = (1-u) \int_\Omega \Big( \int_0^1 F^{-1}_0(s,\omega) ds \Big)  d\nu(\omega) + u \int_\Omega \Big( \int_0^1 F^{-1}_1(s,\omega)ds \Big)  d\nu(\omega) =0.
		\end{aligned}
	\end{align*}
	Next, we set
	\[ \omega \in \Supp(\nu) \quad \mbox{and} \quad x \in \Supp(p(\cdot,\omega)). \]
	Then we have
	\[ x \in \overline{F^{-1}((0,1),\omega)}. \]
	Thus, for every $n\geq0$, there exist $s_n \in (0,1)$ and $s \in [0,1]$ such that 
	\[
	\lim_{n\to+\infty} F^{-1}(s_n,\omega)=x \quad \text{and} \quad \lim_{n\rightarrow +\infty}s_n = s.
	\]
	Since $F^{-1}$ is monotone-increasing, we can take the sequence $\{s_n\}$ to be monotone.
	We set 
	\[  x_{0,n} := F^{-1}_0(s_n, \omega) \quad \mbox{and} \quad x_{1,n} := F^{-1}_1(s_n,\omega). \]
	Then, we have
	\[
	x_n := F^{-1}(s_n,\omega)= (1-u) F^{-1}_0(s_n,\omega) + u F^{-1}_1(s_n,\omega)= (1-u) x_{0,n}+ u x_{1,n}.
	\]
	Since $F^{-1}_0(\cdot,\omega)$ and $F^{-1}_1(\cdot, \omega)$ are left continuous and have right limits, the monotonicity of the sequence $\{s_n\}$ yields that the limits $\lim_{n\rightarrow +\infty}x_{0,n}$ and $\lim_{n\rightarrow +\infty}x_{1,n}$ exist. Thus, we have 
	\[
	x_{0,n} \in \Supp(p_0(\cdot,\omega)), \ \ x_{1,n} \in \Supp(p_1(\cdot,\omega)) \ \ \text{and} \ \ \lim_{n\to+\infty} (1-u) x_{0,n} + u x_{1,n} = x.
	\]
	Similarly, if 
	\[ \omega' \in \Supp(\nu) \quad \mbox{and} \quad x' \in \Supp(p(\cdot,\omega')), \]
	then  there exist $x'_{0,n}$ and $x'_{1,n}$ such that
	\[
	x'_{0,n} \in \Supp(p_0(\cdot,\omega')), \ \ x'_{1,n} \in \Supp(p_1(\cdot,\omega')) \ \ \text{and} \ \ \lim_{n\to+\infty} (1-u) x'_{0,n} + u x'_{1,n} = x'.
	\]
	Since 
	\[ |x_{0,n}-x'_{0,n}| \leq \mathcal D \quad \mbox{and} \quad |x_{1,n}-x'_{1,n}|\leq\mathcal D \quad \mbox{for all $n \in \mathbb{N}$}, \]
	we have 
	\[ |x-x'|\leq\mathcal D. \] Now, it remains to prove the Lipschitz property.   Using the Minkowski inequality to 
	\begin{gather*}
	F^{-1}(s,\omega) -F^{-1}(s,\omega') = g_0(s) + g_1(s), \\
	g_0(s) := (1-u) ( F_0^{-1}(s, \omega) - F_0^{-1}(s,\omega') ), \quad g_1(s) := u ( F_1^{-1}(s, \omega) - F_1^{-1}(s,\omega') ),
	\end{gather*}
	one obtains
	\begin{align*}
		W_2(p(\cdot&,\omega), p(\cdot,\omega')) = \Big[ \int_0^1 | F^{-1}(s,\omega) -F^{-1}(s,\omega')|^2\, ds \Big]^{1/2} \\
		&= \| g_0 + g_1 \|_{L^2} \leq  \|g_0\|_{L^2} + \|g_1\|_{L^2} \\
		&\leq (1-u) \Big[ \int_0^1 | F^{-1}_0(s,\omega) -F^{-1}_0(s,\omega')|^2\, ds \Big]^{1/2} + u \Big[ \int_0^1 | F^{-1}_1(s,\omega) -F^{-1}_1(s,\omega')|^2\, ds \Big]^{1/2}  \\
		&= (1-u) \, W_2(p_0(\cdot,\omega),p_0(\cdot,\omega')) + u \, W_2(p_1(\cdot,\omega),p_1(\cdot,\omega')) \\
		&\leq L|\omega -\omega'|.
	\end{align*}
	This yields that the set $\mathcal X'_\nu(\mathcal D,L)$ is geodesically convex.\\
	
	\noindent {\it Proof of item (\ref{Item:CAT0CompactConvex_03}).} We first show that the set $\mathcal X'_\nu(\mathcal D,L)$ is closed. Let $\{p_n\}_{n\geq0}$ be a sequence of $\mathcal X'_\nu(\mathcal D,L)$ that converges to $p$ in the metric $d_2$. For every $n\geq0$ we denote by $\pi_n(\cdot,\cdot,\omega)$ the unique optimal transport plan between $p_n(\cdot,\omega)$ and $p(\cdot,\omega)$  given by
	\[
	\pi_n(\cdot,\cdot,\omega) = (F^{-1}_{p_n} \otimes F^{-1}_p)( \cdot, \omega)_\sharp \mathcal{L},
	\]
	where $F^{-1}_{p_n}$ and $F^{-1}_{p}$ are pseudo-inverses of cumulative distribution functions $F_{p_n}$ and $F_p$ of $p_n$ and $p$, respectively. Here, for each $s \in [0,1]$, we define
	\[
	\ (F^{-1}_{p_n} \otimes F^{-1}_p)(s, \omega) = (F^{-1}_{p_n}(s,\omega),F^{-1}_p(s,\omega)) \in \mathbb{R}^2.
	\]
	Then, we use the Cauchy-Schwarz inequality to find 
	\begin{align}\label{E3.2.1}
		\begin{aligned}
			&\Big| \int_\Omega \int_\mathbb{R} x \, p_n(dx,\omega) \nu(d\omega) - \int_\Omega \int_\mathbb{R} x \,  p(dx,\omega) \nu(d\omega) \Big| \\
			& \hspace{1cm} = \Big| \int_\Omega \int_\mathbb{R} (x-x') \, \pi_n(dx,dx',\omega) d\nu(\omega) \Big| \\
			& \hspace{1cm} \leq \Big[\int  \Big( \int |x-x'|^2 \, \pi_n(dx,dx',\omega) \Big) d\nu(\omega) \Big]^{1/2} \\
			& \hspace{1cm} = \Big[ \int W_2(p_n(\cdot, \omega), p(\cdot,\omega))^2 d\nu(\omega) \Big]^{1/2} \\
			& \hspace{1cm} =d_2(p_n,p).
		\end{aligned}
	\end{align}
	Since $d_2(p_n,p) \rightarrow 0$, we have
	\[ \int_\Omega \Big( \int_\mathbb{R} x \, p(dx, \omega) \Big) \, \nu(d\omega)=0. \]
	Moreover, it follows from the last equality in \eqref{E3.2.1} that   the sequence of functions $\{W_2(p_n,p)\}_{n\geq0}$ converges to zero in $L^2(\Omega)$. Then, one can find a subsequence $(n_k)_{k\geq0}$ such that, for $\nu$-a.e. $\omega$, 
	\[
	\lim_{k\to+\infty} W_2(p_{n_k}(\cdot,\omega),p(\cdot,\omega)) = 0.
	\] 
	Since the notion of convergence in the 2-Wasserstein space is stronger than the notion of convergence in the  weak-$*$ topology, we have
	\[
	\lim_{k\to+\infty} p_{n_k}(\cdot,\omega) = p(\cdot,\omega), \quad \mbox{in weak-$*$-sense for a.e. $\omega$}.
	\]
	This implies that every $x$ in the support of $p(\cdot,\omega)$ is the limit of a sequence $(x_k)_{k\geq0}$ of points $x_k$ in the support of  $p_{n_k}(\cdot,\omega)$. As each $\omega \mapsto p_{n_k}(\cdot,\omega)$ is uniformly Lipschitz with respect to the 2-Wasserstein distance, the limit $\omega \mapsto p(\cdot,\omega)$ is also Lipschitz a.e. and by continuity, $p$ admits a Lipschitz extension everywhere. It follows that $p \in \mathcal X'_\nu(\mathcal D,L)$, and $\mathcal X'_\nu(\mathcal D,L)$ is closed.	
	
	Next, we derive the compactness property. Let $\{p_n\}_{n\geq0}$ be a sequence in $\mathcal X'_\nu(\mathcal D,L)$. Then $\{p_n\}_{n\geq0}$ is a uniformly Lipschitz sequence in $C(\Omega,\mathcal{P}_2([-\mathcal D,\mathcal D]))$. As $\mathcal{P}_2([-\mathcal D,\mathcal D])$ equipped with the metric $W_2$ is compact, by Arzel\`a-Ascoli theorem, there exists a subsequence $(p_{n_k})_{k\geq0}$ and $p$ in $C^0(\Omega,\mathcal{P}_2([-\mathcal D,\mathcal D]))$ such that
	\begin{gather*}
		\lim_{k\to+\infty} \sup_{\omega \in\Omega} W_2(p_{n_k}(\cdot,\omega),p(\cdot,\omega)) =0.
	\end{gather*}
	In particular, we have
	\[  \lim_{k\to+\infty} d_2(p_{n_k},p)=0. \]
\end{proof}

\section{Dispersion estimate of the characteristics}\label{sec:4}
\setcounter{equation}{0}
In this section, we partially show the positive invariance of $\mathcal C_\nu(\Delta,\tilde\Delta)$ defined in \eqref{C_nu_def}. To achieve that goal, we decompose $\mathcal C_\nu(\Delta,\tilde\Delta)$  into the intersection of two subsets:
\begin{align}\label{C_sub}
\begin{aligned}
    \mathcal{C}_\nu^{(1)}(\Delta) &:= \left\{ \mu\in \mathcal P_2^\nu(\mathbb R^2)~:~\diam_x(\Supp(\mu)) \leq \Delta(x_c[\mu]) \right\},\\
    \mathcal{C}_\nu^{(2)}(\tilde\Delta) &:= \left\{ \mu\in \mathcal P_2^\nu(\mathbb R^2)~:~W_2(\mu(\cdot,\omega), \mu(\cdot,\omega')) \leq\tilde\Delta(x_c[\mu])|\omega - \omega'|, \quad \forall~ \omega, \omega' \in \Omega \right\},
\end{aligned}
\end{align}
where $\Delta$ and $\tilde\Delta$ are $2\pi$-periodic functions as we mentioned in Subsection \ref{sec:2.6}.

Our focus in this section is on establishing the positive invariance of $\mathcal{C}_\nu^{(1)}(\Delta)$ by deriving a dispersion estimate of the characteristics flow. The crucial estimate is a uniform bound on the spatial dispersion of the measure along the flow. Specifically, under suitable conditions, we show that the diameter of the support of the measure evolving under \eqref{A-1} remains uniformly bounded in time. To obtain this result, we analyze the characteristic flow and derive appropriate bounds.

For a given measure-valued solution $\mu$ to \eqref{A-1}, we define a flow parametrized  by $\omega$ given by $(t,x,\omega) \in [0,+\infty) \times \mathbb{R} \times \Omega \mapsto (X[\mu_t](t,x,\omega) = X(t,x,\omega), \omega) \in \mathbb{R} \times \Omega$ as the unique solution to the following Cauchy problem for the ODE:
\begin{equation} 
	\begin{cases} \label{char}
		\displaystyle \frac{d}{dt}X(t, x,\omega) = F[\mu_t](X(t, x,\omega),\omega), \quad \forall~t \ge 0, \\
		\displaystyle X(0, x,\omega) = x.
	\end{cases}
\end{equation}
Then, thanks to the representation formula of the continuity equation (see \cite{A-G-S} for instance), the solution to \eqref{A-1} satisfies the fixed point property:
\[
\forall\, t\geq0, \ \mu_t = (X[\mu_t](t) \otimes \Id)_\sharp \mu_0,
\]
where $X[\mu_t](t) \otimes \Id : \mathbb{R} \times \Omega \rightarrow \mathbb{R} \times \Omega$ is defined by 
\[
X[\mu_t](t) \otimes \Id(x,\omega) := (X[\mu_t](t,x,\omega) ,\omega).
\]

\begin{lemma}\label{L4.3}
	Let $\mu \in {\mathcal C}([0,\infty); \mathcal{P}^\nu_2(\mathbb{R}^2))$ be a global measure-valued solution to \eqref{A-1}, and let $(a,\omega)$ and $(a',\omega')$ be arbitrary points in $\mathrm{supp}(\mu_0)$. If there exist $t_*>0$ and $D>0$ such that for $t\in[0,t_*]$
	\begin{equation}\label{L3.3.0}
		\Big|X(t, a, \omega) - X(t, a', \omega')\Big| \leq D,
	\end{equation}
	then we have
	\[
	\frac{d}{dt} \Big ( X(t, a, \omega) - X(t, a', \omega') \Big) \leq \partial_x F[\tilde{\mu}_t](x_c(t),\omega_c) (X(t,a, \omega) - X(t,a', \omega'))+ C_1(\kappa,\gamma,D),
	\]
	where $C_1(\kappa,\gamma,D)$ is the positive constant defined in \eqref{const}.
\end{lemma}
\begin{proof} We split the proof into several steps. Recall that $\tilde \mu$ has been defined in \eqref{Equation:LeaderDistribution}.

\vspace{0.2cm}
	
	\noindent $\bullet$~Step A (Temporal evolution of $X(t, \cdot, \cdot)$): Let $(\tilde{a}, \tilde{\omega})$ be an arbitrary point in $\mathrm{supp}(\mu_0)$. We have
	\begin{align}
		\begin{aligned} \label{L3.3.1}
			&\frac{d}{dt}X(t,\tilde{a}, \tilde{\omega}) =  F[\mu_t](X(t, \tilde{a}, \tilde{\omega}),{\tilde \omega})  \\
			& \hspace{0.2cm}  = F[\mu_t](x_c(t),\omega_c) + \partial_x F[\mu_t](x_c(t),\omega_c) (X(t,\tilde{a}, \tilde{\omega}) - x_c(t))  \\
			& \hspace{0.7cm}  + F[\mu_t](X(t,\tilde{a}, \tilde{\omega}), \tilde{\omega}) - F[\mu_t](x_c(t),\omega_c) -\partial_x F[\mu_t](x_c(t),\omega_c) (X(t,\tilde{a}, \tilde{\omega}) - x_c(t))\\
			&\hspace{.2cm}=: F[\mu_t](x_c(t),\omega_c) + \partial_x F[\mu_t](x_c(t),\omega_c) (X(t,\tilde{a}, \tilde{\omega}) - x_c(t))  + \mathcal{R}(t,\tilde{a},\tilde{\omega}).
		\end{aligned}
	\end{align}
	We rewrite the remainder term ${\mathcal R}$ as 
	\begin{align*}
		\begin{aligned}
			\mathcal{R}(t,\tilde{a},\tilde{\omega}) =& F[\mu_t](X(t,\tilde{a}, \tilde{\omega}), \tilde{\omega}) - F[\mu_t](X(t,\tilde{a}, \tilde{\omega}), w_c) \\
			& + F[\mu_t](X(t,\tilde{a}, \tilde{\omega}), w_c) - F[\mu_t](x_c(t),\omega_c) \\
			& -\partial_x F[\mu_t](x_c(t),\omega_c) (X(t,\tilde{a}, \tilde{\omega}) - x_c(t)).
		\end{aligned}
	\end{align*}
	Then, we use $(\mathcal{A}_2)$, Taylor expansion and \eqref{L3.3.0} to find 
	\begin{equation}\label{L3.3.2}
		\|\mathcal{R}\|_\infty \leq\gamma I + \frac{\kappa I}{2}|X(t,\tilde{a}, \tilde{\omega}) - x_c(t)|^2 \leq  I\gamma + \frac{\kappa I}{2}D^2.
	\end{equation}
	
	\vspace{0.2cm}
	
	\noindent $\bullet$~Step B (Temporal evolution of difference between two characteristics): It follows from \eqref{L3.3.1} that 
	\begin{align}\label{L3.3.3}
		\begin{aligned}
			&\frac{d}{dt}X(t, a, \omega) - \frac{d}{dt}X(t, a', \omega') \\
			& \hspace{1cm} =\partial_x F[\mu_t](x_c(t),\omega_c) (X(t,a, \omega) - X(t,a', \omega')) +  \mathcal{R}(t,a,\omega) - \mathcal{R}(t,a',\omega')\\
			&  \hspace{1cm} = \partial_x F[\tilde{\mu}_t](x_c(t),\omega_c) (X(t,a, \omega) - X(t,a', \omega'))\\
			&\hspace{1.5cm}+(\partial_x F[\mu_t](x_c(t),\omega_c) - \partial_x F[\tilde{\mu}_t](x_c(t),\omega_c)) (X(t,a, \omega) - X(t,a', \omega'))\\
			&\hspace{1.5cm}+\mathcal{R}(t,a,\omega) - \mathcal{R}(t,a',\omega').
		\end{aligned}
	\end{align}
	On the other hand, it follows from $(\mathcal{A}_1)$  that
	\[
	\Big|\partial_x F[\mu_t](x_c(t),\omega_c) - \partial_x F[\tilde{\mu}_t](x_c(t),\omega_c)\Big| \leq \kappa MW_1(\mu_t, \tilde{\mu}_t).
	\]
	Let $\gamma_t$ be the deterministic coupling of $(\mu_t, \tilde{\mu}_t)$ induced by the map $T_t : \mathbb{R}^2 \rightarrow \mathbb{R}^2$ defined by
	\[ T_t(x,\omega) = (x_c(t), \omega). \]
	Then, we have
	\begin{align}\label{L3.3.3.1}
		\begin{aligned}
			W_1(\mu_t, \tilde{\mu}_t) &\leq \int_{\mathbb{R}^2 \times \mathbb{R}^2}\Big(|x-y| + |\omega - \omega_*|\Big) \gamma_t (dx,dy,d\omega,d\omega)\\
			&=\int_{\mathbb{R}^2}|x-x_c(t)| \mu_t(dx,d\omega) \leq D,
		\end{aligned}
	\end{align}
	which implies
	\begin{align}\label{L3.3.4}
		\Big|\partial_x F[\mu_t](x_c(t),\omega_c) - \partial_x F[\tilde{\mu}_t](x_c(t),\omega_c)\Big| \le\kappa MD.
	\end{align}
	Finally, we combine \eqref{L3.3.0}, \eqref{L3.3.2}, \eqref{L3.3.3} and \eqref{L3.3.4} to see 
	\begin{align*}
		\begin{aligned}
			&\frac{d}{dt}X(t, a, \omega) - \frac{d}{dt}X(t, a', \omega') \\
			& \hspace{1cm} \leq \partial_x F[\tilde{\mu}_t](x_c(t),\omega_c) (X(t,a, \omega) - X(t,a', \omega'))+\kappa MD^2 +2I\gamma + \kappa ID^2.
		\end{aligned}
	\end{align*}
	This ends the proof.
\end{proof}

\vspace{0.2cm}
Let $\mu \in {\mathcal C}([0,\infty); \mathcal{P}_2^\nu(\mathbb{R}^2))$ be a measure-valued solution to \eqref{A-1}. We introduce a change of variable $s = x_c[\mu_t]=x_c(t)$, assuming that there exists $D>0$ and $t_*>0$ such that
\begin{align*}
    \diam_x(\Supp(\mu_t))\le D,\quad\forall~t\in[0,t_*].
\end{align*}

By ($\mathcal A_4$), Lemma \ref{L4.2} and \eqref{L3.3.3.1}, we obtain the following lower bound on $\dot x_c(t)$:
\begin{align*}
    \dot x_c(t)\ge A-((\kappa+1)I+ \kappa M)D-I\gamma,\quad\forall~t\in[0,T].
\end{align*}
To ensure that $\dot x_c(t)>0$ on $[0,t_*]$, we assume
\begin{equation}\label{param1}
A - ((\kappa+1)I+ \kappa M)D - I \gamma   > 0.
\end{equation}

Under this assumption, the function $t \in[0,t_*] \mapsto x_c(t) \in\mathbb{R}$ is  strictly increasing  and  admits the inverse function $\tau[\mu]: [s_0, s_*] \rightarrow [0,t_*]$ where $s_0:=x_c(0)$ and $s_*:=s(t_*)$.  Using the  change of variables $t=\tau[\mu](s)$ the  characteristic flow   \eqref{char} becomes:
\begin{equation} \label{D-5}
\begin{aligned}
&t=\tau[\mu]\circ x_c[\mu_t], \\
&\tilde{X}(s,x,\omega) := X(\tau[\mu](s), x,\omega), \quad \forall~(x,\omega) \in \mathrm{supp}(\mu_0).
\end{aligned}
\end{equation}
In the next lemma, we study the dispersion estimate  of $\tilde X$ in $s$-variable. 
\begin{lemma}\label{L4.4} 
Suppose that the parameters $\kappa, \gamma$ and $D$ satisfy  \eqref{param1}.  Let $\mu \in {\mathcal C}([0,\infty); \mathcal{P}_2^\nu(\mathbb{R}^2))$ be a measure-valued solution to \eqref{A-1} satisfying a priori condition that there exists $t_* > 0$ such that
\begin{align*}
    \diam_x(\mu_t)\le D,\quad\forall~t\in[0,t_*].
\end{align*}
Then, using the change of variable \eqref{D-5}, one obtains for all $s\in(s_0,s_*)$, 
\[
\frac{d}{ds}\Big(\tilde{X}(s,x,\omega)-\tilde{X}(s,x',\omega')\Big) \leq \frac{\partial_x F[\tilde{\mu}_t]( x_c(t), \omega_c)}{F[\tilde{\mu}_t](x_c(t), \omega_c)}\Big(\tilde{X}(s,x,\omega)-\tilde{X}(s,x',\omega')\Big) + C_2(\kappa,\gamma,D),
\]
for all $(x,\omega),(x',\omega')\in\Supp(\mu_0)$, where the constant $C_2(\kappa,\gamma,D)$ is defined in \eqref{const} and $\tilde \mu$ is defined in \eqref{Equation:LeaderDistribution}.
\end{lemma}
\begin{proof} We use the chain rule \eqref{L3.3.1} to see
\begin{align*}
\frac{d}{dt}X(t,x,\omega) &= \frac{d}{dt}\tilde{X}(s,x,\omega) = \frac{d}{ds}\tilde{X}(s,x,\omega)\frac{d}{dt}x_c(t)\\
&= \frac{d}{ds}\tilde{X}(s,x,\omega)\bigg(\frac{d}{dt}x_c(t) - F[\tilde{\mu}_t](x_c(t), \omega_c)\bigg) + \frac{d}{ds}\tilde{X}(s,x,\omega)\bigg(F[\tilde{\mu}_t](x_c(t), \omega_c)\bigg). 
\end{align*}
By \eqref{As4}, one has $F(\tilde\mu_t,x_c(t),\omega_c)>0$ and use it to derive
\begin{align*}
\frac{d}{ds}\tilde{X}(s,x,\omega) &= \frac{\frac{d}{dt}X(t,x,\omega) }{F[\tilde{\mu}_t]( x_c(t), \omega_c)} - \frac{\frac{d}{ds}\tilde{X}(s,x,\omega)\Big(\frac{d}{dt}x_c(t) - F[\tilde{\mu}_t](x_c(t), \omega_c)\Big)}{F[\tilde{\mu}_t](x_c(t), \omega_c)}
\end{align*}
for all $t\in(0,t_*)$ and $s\in(s_0,s_*)$. This yields
\begin{align}
\begin{aligned} \label{D-6}
&\hspace{-1cm}\frac{d}{ds}\Big(\tilde{X}(s,x,\omega)-\tilde{X}(s,x',\omega')\Big) \\
& = \frac{\frac{d}{dt}\Big(X(t,x,\omega)-X(t,x',\omega')\Big) }{F[\tilde{\mu}_t](x_c(t), \omega_c)}\\
&\hspace{.4cm} - \frac{\frac{d}{ds}\Big(\tilde{X}(s,x,\omega)- \tilde{X}(s,x',\omega')\Big)\Big(\frac{d}{dt}x_c(t) - F[\tilde{\mu}_t](x_c(t), \omega_c)\Big)}{F[\tilde{\mu}_t](x_c(t), \omega_c)}\\
&=: {\mathcal I}_{21}  + {\mathcal I}_{22}.
\end{aligned}
\end{align}
Below, we estimate the terms ${\mathcal I}_{2i}$ one by one. 

\vspace{.2cm}

\noindent $\bullet$~Case A (Estimate of ${\mathcal I}_{21}$): By Lemma \ref{L4.3}, we have
\begin{equation} \label{D-7}
{\mathcal I}_{21} \leq \frac{\partial_x F[\tilde{\mu}_t]( x_c(t), \omega_c)}{F[\tilde{\mu}_t](x_c(t), \omega_c)}\Big(\tilde{X}(s,x,\omega)-\tilde{X}(s,x',\omega')\Big) + \frac{\kappa( M+I)D^2 +2 I\gamma}{F[\tilde{\mu}_t](x_c(t), \omega_c)}.
\end{equation}

\vspace{0.2cm}

\noindent $\bullet$~Case B (Estimate of ${\mathcal I}_{22}$): It follows from Lemma \ref{L4.1} and $(\mathcal{A}_2)$ that 
\begin{align*}
\begin{aligned} 
& \bigg|\frac{d}{ds}\Big(\tilde{X}(s,x,\omega)- \tilde{X}(s,x',\omega')\Big)\bigg| \\
& \hspace{1cm} = \frac{\Big|\frac{d}{dt}(X(t,x,\omega)-X(t,x',\omega'))\Big|}{\Big|\frac{d}{dt}x_c(t)\Big|} \leq \frac{\kappa ID +\kappa( M+I)D^2 +2I\gamma}{F[\tilde{\mu}_t](x_c(t), \omega_c) -((\kappa+1)I+ \kappa M) D - I \gamma}.
\end{aligned}
\end{align*}
By Lemma \ref{L4.2} and $(\mathcal{A}_4)$, we obtain
\begin{align}
\begin{aligned} \label{D-8}
{\mathcal I}_{22} &\leq \frac{\bigg|\frac{d}{ds}\Big(\tilde{X}(s,x,\omega)- \tilde{X}(s,x',\omega')\Big)\Big|\cdot\bigg|\frac{d}{dt}x_c(t) - \Big(F(\tilde{\mu}_t,x_c(t),\omega_c)\Big)\Big|}{F[\tilde{\mu}_t](x_c(t), \omega_c)}\\
&\leq \frac{\big(\kappa ID +\kappa(M+I)D^2 +2 I\gamma\big)\big(((\kappa+1)I+ \kappa M)D+ I \gamma \big)}{A\big(A - ((\kappa+1)I+ \kappa M)D - I \gamma\big)}.
\end{aligned}
\end{align}
In \eqref{D-6}, we combine \eqref{D-7} and \eqref{D-8} to find the desired estimate:
\begin{align*}
\begin{aligned}
&\frac{d}{ds}\Big(\tilde{X}(s,x,\omega)-\tilde{X}(s,x',\omega')\Big) \\
& \hspace{1cm} \leq \frac{\partial_x F[\tilde{\mu}_t](x_c(t), \omega_c)}{F[\tilde{\mu}_t]( x_c(t), \omega_c)}\Big(\tilde{X}(s,x,\omega)-\tilde{X}(s,x',\omega')\Big) + \frac{\kappa( M+I)D^2 +2I\gamma}{A}\\
& \hspace{1.5cm}+\frac{\big(\big(\kappa ID +\kappa( M+I)D^2 +2 I\gamma\big)((\kappa+1)I+ \kappa M)D+ I \gamma \big)}{A\big(A -((\kappa+1)I+ \kappa M)D - I \gamma\big)},
\end{aligned}
\end{align*}
which completes the proof.
\end{proof}
Now, we are ready to show that $\mathcal C_\nu^{(1)}(\Delta)$ is positively invariant.
\begin{proposition}\label{P4.1}
Suppose that the parameters $\kappa,\gamma$ and $D$ satisfy \eqref{param1} and   
\begin{align}\label{D_cond}
2\pi C_2(\kappa,\gamma,D)\frac{\exp(2\pi\kappa I/A)}{\kappa B}\le D.
\end{align}
Let $s \in\mathbb{R} \mapsto \Delta(s) \in (0,D)$ be the unique $2\pi$-periodic solution of 
\begin{gather*}
\frac{d\Delta(s)}{ds} = C_2(\kappa,\gamma,D) + \frac{\partial_s F[\delta_s \otimes \nu](s, \omega_c) }{F[\delta_s \otimes \nu]( s, \omega_c)} \Delta(s).
\end{gather*}
Then, the set $\mathcal C_\nu^{(1)}(\Delta)$ is positively invariant under \eqref{A-1}. In particular, if $\mu_0\in\mathcal C_\nu^{(1)}(\Delta)$, then the support of the solution $\mu$ remains uniformly bounded in $x$, i.e.,
    \begin{align*}
        \diam_x(\Supp(\mu_t)) \leq \Delta(x_c[\mu_t]) \le D,\quad\forall~t\ge0.
    \end{align*}
\end{proposition}
\begin{proof}
    This proposition can be interpreted as an extension of Lemma \ref{L4.4}, in the sense that an initial condition in $\mathcal C_\nu^{(1)}(\Delta)$ eventually implies that the a priori assumption of Lemma \ref{L4.4} holds for all $t\ge0$. We split the proof into two steps.

    \vspace{.2cm}
    
    \noindent$\bullet$ Step A: By Lemma \ref{L4.4}, if 
    \begin{align*}
        \diam_x(\mu_t)\le D,\quad\forall~t\in[0,t_*],
    \end{align*}
    then $\dot x_c(t)>0$ on $t\in[0,t_*]$ and
    \begin{align*}
        \frac{d}{ds}\Big(\tilde X(s,x,\omega)-\tilde X(s,x',\omega')\Big)\le\alpha+\beta(s)\Big(\tilde X(s,x,\omega)-\tilde X(s,x',\omega')\Big),
    \end{align*}
    for all $s\in(s_0,s_*)$, where we define
    \begin{align*}
        \alpha=C_2(\kappa,\gamma,D)\quad\mbox{and}\quad\beta(s)=\frac{\partial_x F[\tilde{\mu}_t](x_c(t), \omega_c) }{F[\tilde{\mu}_t]( x_c(t), \omega_c)} = \frac{\partial_s F[\delta_s \otimes \nu](s, \omega_c) }{F[\delta_s \otimes \nu]( s, \omega_c)}.
    \end{align*}
    
    Now, consider the differential inequality
    \begin{align}\label{DE}
        \frac{dy(s)}{ds}=\alpha+\beta(s)y(s).
    \end{align}
    By ($\mathcal A_5$), the function $\beta$ satisfies \eqref{B-11}. Then, by Lemma \ref{L2.2}, the differential equation \eqref{DE} admits a (unique) positive periodic solution to \eqref{DE}, denoted by $\Delta(s)$. Since $\mu_0\in \mathcal{C}_\nu^{(1)}(\Delta)$, it follows from the definition of $\mathcal{C}_\nu^{(1)}(\Delta)$ that
    \begin{align*}
        \tilde X(s_0,x,\omega)-\tilde X(s_0,x',\omega')\le\Delta (s_0),\quad\forall~(x,\omega),(x',\omega')\in\Supp(\mu_0).
    \end{align*}
    Then, by the comparison principle, we have
    \begin{align*}
        \tilde X(s,x,\omega)-\tilde X(s,x',\omega')\le\Delta(s)\le D,\quad\forall~s\in[s_0,s_*],
    \end{align*}
    where the last inequality follows from the condition \eqref{D_cond} and the upper bound of $\Delta(s)$ obtained by Lemma \ref{L2.2}. 
Indeed using the estimate \eqref{Equation:L2.2_01}, $(\mathcal{A}_2)$, and $(\mathcal{A}_4)$, one obtains
\begin{gather*}
\int_0^{2\pi} |\beta(s)| \, ds \leq 2\pi \frac{\kappa I}{A}, \quad  \int_0^{2\pi} \beta(s) \, ds \leq - \kappa B, \\
\max_{s \in \mathbb{R}} \Delta(s) \leq 2 \pi C_2(\kappa,\gamma,D) \frac{\exp(2\pi \kappa I/A)}{\kappa B} \leq D.
\end{gather*}

    \vspace{.2cm}
    
    \noindent$\bullet$ Step B: Let $\Sigma(s)=\diam_x(\Supp(\mu_t))$, where $s=x_c(t)$. For any $\varepsilon>0$, define
    \begin{align*}
        s^\infty(\varepsilon):=\inf\{s>s_0~:~\Sigma(s)<\Delta(s)+\varepsilon\}.
    \end{align*}
    Assume that $\mu_0\in\mathcal C_\nu^{(1)}(\Delta)$. Since $\Sigma(s)$ is continuous and $\Sigma(s_0)\le\Delta(s_0)$, one can apply a bootstrapping argument using the comparison principle and Step A to conclude that
    \begin{align*}
        s^\infty(\varepsilon)=\infty,\quad\forall~\varepsilon>0.
    \end{align*}
    In other words, one has
    \begin{align*}
        \Sigma(s)<\Delta(s)+\varepsilon,\quad\forall~s\ge s_0,\quad\forall~\varepsilon>0.
    \end{align*}
    Taking $\varepsilon\to0$, we conclude that
    \begin{align*}
        \Sigma(s)\le\Delta(s),\quad\forall~s\ge s_0,
    \end{align*}
    which implies the desired uniform-in-time bound on $\diam_x(\Supp(\mu_t))$. This ends the proof.
\end{proof}

\begin{remark} \label{Remark:Feasibility}
We check that there exists a non empty open set of  parameters $(\kappa,\gamma)$ satisfying \eqref{param1} such that \eqref{D_cond} admits a solution $D$. In particular  the domain of validity of Theorem \ref{main} contains the subdomain
\begin{gather*}
\{ (\kappa,\gamma) : \kappa \in (0,1) \ \text{and} \ \gamma \in (0,\gamma(\kappa)) \}
\end{gather*} 
where $\gamma(\kappa) = \mathcal{O}(\kappa)$. Moreover, a dispersion constant $D$ can be chosen independently of $(\kappa,\gamma)$, depending only on $A,B,I,M$ on that subdomain. 
\end{remark}

\begin{proof}
Let $\kappa>0$. define
\begin{gather}
D(\kappa) := \min \Big( \frac{A}{4((\kappa+1)I+\kappa M)}, \frac{BA^2}{16\pi((M+I)A +I((\kappa+1)I+ \kappa M))\exp(2\pi I/A)} \Big). \label{Equation:DispersionFormula}
\end{gather}
Choose any $\gamma \in (0,\gamma(\kappa))$ where
\begin{gather}
\gamma(\kappa) := \min \Big( \frac{A}{4I}, \frac{\kappa(M+I)D(\kappa)^2}{2I}, \frac{((\kappa+1)I+\kappa M)D(\kappa)}{I} \Big). \label{Equation:GammaFormula}
\end{gather}
We observe that the assumption \eqref{param1}  is satisfied using
\begin{gather*}
\gamma \leq \frac{A}{4I}, \quad D(\kappa) \leq \frac{A}{4((\kappa+1)I+\kappa M)}, \\
A - ((\kappa+1)I+\kappa M)D - I \gamma \geq \frac{A}{2} >0.
\end{gather*}
Moreover, from \eqref{Equation:GammaFormula} we have
\begin{gather}
\gamma \leq  \frac{\kappa(M+I)D(\kappa)^2}{2I}, \notag 
\end{gather}
which yields
\begin{equation}\label{Equation:C1}
C_1(\kappa,\gamma,D) \leq 2\kappa(M+I)D(\kappa)^2.
\end{equation}
Next, using \eqref{Equation:C1} and 
\begin{equation*}
\gamma \leq \frac{((\kappa+1)I+\kappa M)D(\kappa)}{I}, 
\end{equation*}
we estimate that
\begin{align*}
C_2(\kappa,\gamma,D(\kappa)) &= \frac{C_1(\kappa,\gamma,D) A + \kappa I D(\kappa) (((\kappa+1)I+\kappa M)D(\kappa) +I\gamma)}{A(A- ((\kappa+1)I+ \kappa M)D - I \gamma)}\\
&\leq \frac{8\kappa ((M+I)A + I((\kappa+1)I+\kappa M))D(\kappa)^2}{A^2}.
\end{align*}
Then, we obtain
\begin{equation*}\label{Equation:C2Formula} 
2\pi C_2(\kappa,\gamma,D)\frac{\exp(2\pi\kappa I/A)}{\kappa B} \leq 16 \pi \frac{(M+I)A + I((\kappa+1)I+\kappa M) \exp(2\pi \kappa I/A)}{BA^2} D(\kappa)^2. 
\end{equation*}

The assumption \eqref{D_cond} is thus satisfied using
\begin{gather*}
D(\kappa) \leq \frac{BA^2}{16\pi((M+I)A +I((\kappa+1)I+\kappa M))\exp(2\pi I/A)}.
\end{gather*}
\end{proof}

\section{Existence of periodic solutions}\label{sec:5}
\setcounter{equation}{0}
In this section, we prove the existence of a periodic solution to \eqref{A-1}, i.e., provide the proof of the second part of Theorem \ref{main}. To this end, we first show the positive invariance of $\mathcal C_\nu(\Delta,\tilde\Delta)$ based on the positive invariance of $\mathcal C_\nu^{(1)}(\Delta)$, which was shown in the previous section. Next, we construct a Poincar\'e map on a Poincar\'e section in $\mathcal C_\nu(\Delta,\tilde\Delta)$ that is Lipschitz continuous, allowing us to apply the Schauder fixed point theorem. This yields the existence of a periodic solution.

Throughout this section, we use two measurement functions. We denote by $\Pi_2(\eta_1,\eta_2)$  the set of couplings between two probability measures $\eta_1,\eta_2 \in\mathcal{P}_2(\mathbb{R})$  supported on $\mathbb{R}$ with finite second moment. Let $\mu$ be a solution of \eqref{A-1} with  initial datum $\mu_0\in\mathcal P_2^\nu(\mathbb R^2)$. For each $\omega_1\ne\omega_2\in\Omega$, there exists the (unique optimal) coupling $\pi_0\in\Pi_2(\mu_0(\cdot,\omega_1),\mu_0(\cdot,\omega_2))$ that minimizes the quadratic cost function $c(x,y)=|x-y|^2$. Using this, we set the following measurements:
\begin{align}\label{E5.1}
\begin{aligned}
\Lambda^{\mu_0}[\omega_1, \omega_2](t) &:= \bigg(\int_{\mathbb{R}^2} |X(t, x, \omega_1) - X(t,y, \omega_2)|^2 \pi_0 (dx, dy)\bigg)^{1/2},\\
\Gamma^{\mu_0}[\omega_1, \omega_2](t) &:= \frac{\Lambda^{\mu_0}[\omega_1, \omega_2](t)}{|\omega_1 - \omega_2|},
\end{aligned}
\end{align}
where $X$ denotes the flow generated by the characteristic ODE \eqref{char} based on the solution $\mu$. Note that $\Lambda$ evaluates the cost using a pushforward measure under the flow map. Specifically, given the map
\begin{align*}
    X(t,\cdot, \omega_1) \otimes X(t, \cdot, \omega_2) : \mathbb{R}^2 \rightarrow \mathbb{R}^2,\quad(x,y) \mapsto (X(t,x, \omega_1),X(t,y, \omega_2)),
\end{align*}
we define the pushforward measure
\[ \pi_t := X(t,\cdot, \omega_1) \otimes X(t, \cdot, \omega_2)_\sharp \pi_0. \]
Then, we have
\begin{align*}
    \Lambda^{\mu_0}[\omega_1,\omega_2](t)=\left(\int_{\mathbb R^2}|x-y|^2\pi_t(dx,dy)\right)^{1/2}.
\end{align*}
It is straightforward to verify that $\pi_0 \in \Pi_2(\mu_0(\cdot,\omega_1), \mu_0(\cdot,\omega_2))$ implies
\[\pi_t \in \Pi_2(\mu_t(\cdot,\omega_1), \mu_t(\cdot,\omega_2)),\quad\forall~t\ge0.\]
However, we remark that even though $\pi_0$ is an optimal transport plan, there is no guarantee that $\pi_t$ remains optimal.

\subsection{Regularity propagation in Wasserstein distance}
\label{sec:5.1}
In this subsection, we show the positive invariance of $\mathcal C_\nu(\Delta,\tilde\Delta)$ and provide preparatory estimates for the proof of item (ii) of \eqref{C_nu_def}. Under assumptions \eqref{param1} and \eqref{D_cond} the conclusions  of Proposition \ref{P4.1}, we use the change of variable \eqref{D-5}
\begin{align*}
    \tilde\Gamma^{\mu_0}[\omega_1,\omega_2](s)=\Gamma^{\mu_0}[\omega_1,\omega_2](t)\quad\mbox{with}\quad s=x_c(t).
\end{align*}
\begin{lemma} \label{L5.1}
Suppose that the parameters $\kappa,\gamma$ and $D$ satisfy \eqref{param1} and \eqref{D_cond}. Let $\Delta$ be the $2\pi$-periodic function given by Proposition \ref{P4.1} and  $\mu\in\mathcal C([0,\infty);\mathcal P_2(\mathbb R^2))$ be a solution of \eqref{A-1} with   initial datum $\mu_0\in\mathcal C_\nu^{(1)}(\Delta)$. The following assertions hold.
    \begin{enumerate}
        \item The time-evolution of $\Gamma$ satisfies
        \begin{align*}
            \frac{d\Gamma^{\mu_0}[\omega_1,\omega_2](t)}{dt} \leq \bigg(\partial_x F[\tilde{\mu}](x_c, \omega_c) + E_1(\kappa,\gamma,D)\bigg)\Gamma^\mu[\omega_1,\omega_2](t) + I,
        \end{align*}
        for all $t\ge0$ and $\omega_1,\omega_2\in\Omega$, where the constant $E_1(\kappa,\gamma,D)$ is defined in \eqref{const} and $\tilde\mu$ is defined in \eqref{Equation:LeaderDistribution}. 
        \vspace{0.2cm}
        \item The time-evolution of $\tilde\Gamma$ satisfies
        \begin{align*}
            \frac{d\tilde\Gamma^{\mu_0}[\omega_1,\omega_2](s)}{ds}&\le\frac{I}{A-\kappa(I+M)D-I\gamma}\\
            &\hspace{.5cm}+\left(\frac{\partial_sF[\delta_s\otimes\nu](s,\omega_c)}{F[\delta_s\otimes\nu](s,\omega_c)}+E_2(\kappa,\gamma,D)\right)\tilde\Gamma^{\mu_0}[\omega_1,\omega_2](s),
        \end{align*}
        for all $s\ge s_0$ and $\omega_1,\omega_2\in\Omega$, where the constant $E_2(\kappa,\gamma,D)$ is defined in \eqref{const}.
    \end{enumerate}
\end{lemma}
\begin{proof}[Proof of item (1)]
For fixed $\omega_1,\omega_2\in\Omega$, we use the simplified notation
\[\Lambda(t)=\Lambda^{\mu_0}[\omega_1,\omega_2](t).\]
By the definition of $\Lambda$ in \eqref{E5.1}, one has
\begin{equation}\label{L4.1-1}
\frac{1}{2}\frac{d\Lambda^2(t)}{dt} = \int_{\mathbb R^2} \Big( X(t,x,\omega_1) - X(t,y,\omega_2) \Big) \left( \frac{d}{dt}X(t,x,\omega_1) - \frac{d}{dt}X(t,y,\omega_2) \right) \, \pi_0(dx,dy).
\end{equation}
For simplicity, we set
\[ 
X_1= X(t,x,\omega_1), \quad X_2 = X(t,y,\omega_2)\quad \mbox{and}\quad x_c = x_c[\mu_t].
\]
Then for fixed $x,y\in\mathbb R$ we have
\begin{equation}\label{L4.1-2}
\frac{dX_1}{dt} - \frac{dX_2}{dt} = F[\mu_t](X_1,\omega_1) - F[\mu_t](X_2,\omega_2) = \partial_x F[\tilde \mu_t](x_c,\omega_c) (X_1-X_2) + R_t,
\end{equation}
where the reminder term $R_t$ can split into four terms:
\begin{align}
\begin{aligned} \label{E-2}
R_t &=   F[\mu_t](X_1,\omega_1) - F[\mu_t](X_1, \omega_2)  \\
&\hspace{.4cm}+  F[\mu_t](X_1,\omega_2) - F[\mu_t](X_2,\omega_2)  - \partial_x F[\mu_t](X_2,\omega_2)  (X_1-X_2)\\
&\hspace{.4cm}+ ( \partial_x F[\mu_t](X_2,\omega_2) - \partial_x F[\mu_t](x_c,\omega_c) ) (X_1-X_2)  \\
&\hspace{.4cm}+ ( \partial_x F[\mu_t](x_c,\omega_c) - \partial_x F[\tilde \mu_t](x_c,\omega_c) )(X_1-X_2)\\
&=: \sum_{i=1}^4 R_t^{(i)}.
\end{aligned}
\end{align}
Now, we estimate $R_t^{(i)}$ for $i=1,2,3,4$ one by one. \newline

\noindent $\bullet$~Case A (Estimate of $R^{(1)}_t$):~By $(\mathcal{A}_2)$ with Taylor expansion, the first term is bounded by
\begin{equation} \label{E-3}
|R^{(1)}_t| \leq I |\omega_1 - \omega_2|.
\end{equation}
\noindent $\bullet$~Case B (Estimate of $R^{(2)}_t$):~Note that there exists some $x_*(t)$ between $X_1(t)$ and $X_2(t)$ such that
\begin{equation*}
R^{(2)}_t =\frac{1}{2} \partial_{x}^2 F[\mu_t](x_*, \omega_2) (X_1-X_2)^2.
\end{equation*}
Then, it follows from $(\mathcal{A}_2)$ and Proposition \ref{P4.1} that 
\begin{equation} \label{E-4}
\big|R^{(2)}_t\big| \leq \frac{1}{2}\kappa I D |X_1-X_2|.
\end{equation}
\noindent $\bullet$~Case C (Estimate of $R^{(3)}_t$):~By $(\mathcal{A}_2)$ and Proposition \ref{P4.1}, we have
\begin{align} 
\begin{aligned} \label{E-5}
\big|R^{(3)}_t\big| &\leq \kappa I \big( | X_1(t)-x_c(t)| + |\omega_1-\omega_c| \big) |X_1-X_2| \leq \kappa I (D+\gamma) |X_1-X_2|.
\end{aligned}
\end{align}
\noindent $\bullet$~Case D (Estimate of $R^{(4)}_t$):~Again, we use $(\mathcal{A}_1)$ to find 
\begin{equation} \label{E-6}
\big|R^{(4)}_t\big| \leq \kappa M W_1(\tilde \mu_t, \mu_t) |X_1-X_2| \leq \kappa MD |X_1-X_2|,
\end{equation}
where we used the bounded dispersion to estimate the 1-Wasserstein distance obtained in \eqref{L3.3.3.1}. In \eqref{E-2}, we combine all the estimates \eqref{E-3}, \eqref{E-4}, \eqref{E-5} and \eqref{E-6} to find 
\begin{align*}
    |R_t|\le I|\omega_1-\omega_2|+\kappa\left(D\left(\frac{3I}{2}+M\right)+I\gamma\right)|X_1-X_2|
\end{align*}
Thus, it follows from \eqref{L4.1-1} - \eqref{L4.1-2}, remembering that $X_1$ and  $X_2$ depend respectively on $x$ and $y$, that for fixed $t$,
\begin{align*}
    \frac{1}{2}\frac{d\Lambda^2(t)}{dt} &= \int_{\mathbb{R}^2}  (X_1-X_2)  \Big( \partial_x F[\tilde{\mu_t}](x_c, \omega_c)(X_1-X_2) + R_t \Big) \pi_0(dx,dy) \\
    &\leq \partial_x F[\tilde{\mu_t}](x_c, \omega_c) \Lambda^2 + \kappa \left(D\left(\frac{3I}{2} +  M\right) +  I\gamma\right)\Lambda^2 \\
    &\quad+  I \int_{\mathbb{R}^2}|X_1-X_2|\cdot |\omega_1 - \omega_2|\pi_0(dx,dy)\\
    &\leq \partial_x F[\tilde{\mu_t}](x_c, \omega_c) \Lambda^2 + \kappa\left(D\left(\frac{3I}{2}+M\right)+I\gamma\right)\Lambda^2 + I|\omega_1 - \omega_2| \Lambda,
\end{align*}
where we use H\"older's inequality in the last estimate. Therefore, we have 
\[
\frac{d\Lambda}{dt}\leq \partial_x F[\tilde{\mu}](x_c, \omega_c) \Lambda + \kappa\left(D\left(\frac{3I}{2}+M\right)+I\gamma\right)\Lambda + I|\omega_1 - \omega_2|.
\]
We divide the above relation by $|\omega_1 - \omega_2|$ to obtain
\begin{gather*}
\frac{d\Gamma}{dt} \leq \bigg(\partial_x F[\tilde{\mu}](x_c, \omega_c) + E_1(\kappa,\gamma,D)\bigg)\Gamma + I.
\end{gather*}

\vspace{0.2cm}

\noindent{\it Proof of item (2).} For fixed $\omega_1,\omega_2\in\Omega$, we use the simplified notation
\[\tilde\Gamma(s)=\tilde\Gamma^{\mu_0}[\omega_1,\omega_2](s).\]
From (1) and Lemma \ref{L4.1}, we estimate as follows:
\begin{align*}
    \begin{aligned}
        \frac{d\tilde \Gamma(s)}{ds} &= \frac{d \Gamma(t)}{dt}\frac{dt}{ds}\\
        &\leq \frac{\partial_x F[\tilde{\mu}_t](x_c, \omega_c)}{dx_c/dt}\Gamma(t)  + \kappa\left(D\left(\frac{3I}{2}+M\right)+I\gamma\right)\frac{\Gamma(t)}{dx_c/dt} + \frac{I}{dx_c/dt}\\
        &=\frac{\partial_sF[\delta_s \otimes \nu](s,\omega_c)}{F[\delta_s \otimes \nu](s,\omega_c)} \tilde{\Gamma}(s)  +\kappa\left(D\left(\frac{3I}{2}+M\right)+I\gamma\right)\frac{\tilde\Gamma(s)}{dx_c/dt} + \frac{I}{dx_c/dt} \\
        &\hspace{0.4cm} + \partial_sF[\delta_s \otimes \nu](s,\omega_c)\bigg(\frac{1}{dx_c/dt} - \frac{1}{F[\delta_s \otimes \nu](s,\omega_c)} \bigg) \tilde{\Gamma}(s),
    \end{aligned}
\end{align*}
where we used 
\[ \tilde{\mu}_t = \delta_{x_c(t)}\otimes \nu = \delta_s \otimes \nu. \] 
By $(\mathcal{A}_2)$, $(\mathcal{A}_4)$ and Lemma \ref{L4.2}, i.e.,
\begin{align*}
    \frac{dx_c(t)}{dt}\geq A-((\kappa+1)I+\kappa M)D-I\gamma,
\end{align*}
we have 
\begin{align*}
    \begin{aligned}
        \partial_sF[\delta_s \otimes \nu](s,\omega_c)\left(\frac{1}{dx_c/dt} - \frac{1}{F[\delta_s \otimes \nu](s,\omega_c)} \right)\leq \frac{\kappa I\big(\kappa(I+M)D+I\gamma\big)}{A\big(A-((\kappa+1)I+\kappa M)D-I\gamma\big)}.
    \end{aligned}
\end{align*}
Therefore, we have
\begin{align*}
    \frac{d\tilde\Gamma(s)}{ds}&\leq\left(\frac{\partial_sF[\delta_s\otimes\nu](s,\omega_c)}{F[\delta_s\otimes\nu](s,\omega_c)}+E_2(\kappa,\gamma,D)\right)\tilde\Gamma(s)+\frac{I}{A-((\kappa+1)I+\kappa M)D-I\gamma},
\end{align*}
which is the desired result.
\end{proof}
In the following proposition, we show the set $\mathcal C_\nu(\Delta,\tilde\Delta)=\mathcal C_\nu^{(1)}(\Delta)\cap\mathcal C_\nu^{(2)}(\tilde\Delta)$ (see \eqref{C_sub}) is positively invariant under \eqref{A-1}.
\begin{proposition}\label{P5.1}
Suppose that the parameters $\kappa,\gamma$, $D$ satisfy \eqref{param1}, \eqref{D_cond}, and \eqref{Equation:LipschitzRegularityAssumption}. Let $\Delta$ be the $2\pi$-periodic function given by Proposition \ref{P4.1}. Then there exists a $2\pi$-periodic function $\tilde \Delta : \mathbb{R} \to (0,+\infty)$ so that the set $\mathcal C_\nu(\Delta,\tilde\Delta)$ is positively invariant under \eqref{A-1}.
\end{proposition}
\begin{proof}
Let $\kappa,\gamma,D$ satisfying the assumptions \eqref{param1}, \eqref{D_cond},  \eqref{Equation:LipschitzRegularityAssumption} and $\Delta$ be the $2\pi$-periodic function given by Proposition \ref{P4.1}. As $\mathcal C_\nu(\Delta,\tilde\Delta) =\mathcal C_\nu^{(1)}(\Delta) \cap \mathcal C_\nu^{(2)}(\tilde\Delta)$ and the positive invariance of $\mathcal C_\nu^{(1)}(\Delta)$ has been shown in Proposition \ref{P4.1}, we aim to construct $\tilde \Delta$ and show the positive invariance of $\mathcal C_\nu^{(2)}(\tilde\Delta)$. Therefore, it remains to show  for all $\omega_1\ne\omega_2\in\Omega$, that 
    \begin{align*}
    W_2(\mu_t(\cdot,\omega_1),\mu_t(\cdot,\omega_2))\le\tilde\Delta(x_c[\mu_t])|\omega_1-\omega_2|,\quad\forall~t>0,
    \end{align*}
provided the initial datum $\mu_0\in\mathcal C_\nu^{(1)}(\Delta)$ satisfies
    \begin{align*}
        W_2(\mu_0(\cdot,\omega_1),\mu_0(\cdot,\omega_2))\le\tilde\Delta(x_c[\mu_0])|\omega_1-\omega_2|,
    \end{align*}
From the result in Lemma \ref{L5.1}, we consider the following differential inequality
    \begin{align*}
        \frac{d\tilde\Gamma^{\mu_0}[\omega_1,\omega_2](s)}{ds}\le\tilde\alpha+\tilde\beta(s)\tilde\Gamma^{\mu_0}[\omega_1,\omega_2](s),
    \end{align*}
    where we set
    \begin{align*}
        \tilde\alpha:=\frac{I}{A-\kappa(I+M)D-I\gamma}\quad\mbox{and}\quad\tilde\beta(s):=\frac{\partial_sF[\delta_s\otimes\nu](s,\omega_c)}{F[\delta_s\otimes\nu](s,\omega_c)}+E_2(\kappa,\gamma,D).
    \end{align*}
To unsure that $\int_0^{2\pi}\tilde \beta(s)  \, ds < 0$, we assume
\begin{gather}
E_2(\kappa,\gamma,D)  < \kappa B. \label{Equation:LipschitzRegularityAssumption}
\end{gather}
Let $\tilde\Delta(s)$ be the unique $2\pi$-periodic solution of
    \begin{align*}
        y(s)=\tilde\alpha+\tilde\beta(s)y(s),
    \end{align*}
which is given explicitly in Lemma \ref{L2.2} by the formula \eqref{B-14-2} 
    \begin{align*}
        \tilde\Delta(s)=\frac{\tilde\alpha\int_{s_0}^{s_0+2\pi}\exp\left(\int_{\tilde s}^{s_0}\beta(\tau)d\tau\right)d\tilde s}{\exp\left(-\int_0^{2\pi}\beta(\tilde s)d\tilde s\right)-1},\quad\forall~s\ge s_0=x_c(\mu_0).
    \end{align*}
The comparison principle implies
    \begin{align}\label{tilGD}
        \tilde\Gamma^{\mu_0}[\omega_1,\omega_2](s)\le\tilde\Delta(s),\quad\forall~s\ge s_0,~~\forall~\omega_1,\omega_2\in\Omega,
    \end{align}

    By the definition of $\tilde\Gamma^{\mu_0}[\omega_1,\omega_2]$,
    one has
    \begin{align}\label{W_Gamma}
        \frac{W_2(\mu_t(\cdot,\omega_1),\mu_t(\cdot,\omega_2))}{|\omega_1-\omega_2|}\le\tilde\Gamma^{\mu_0}[\omega_1,\omega_2](x_c[\mu_t]),\quad\forall~\omega_1\ne\omega_2\in\Omega.
    \end{align}
    Therefore, we combine \eqref{tilGD} and \eqref{W_Gamma} to obtain
    \begin{align*}
        \frac{W_2(\mu_t(\cdot,\omega_1),\mu_t(\cdot,\omega_2))}{|\omega_1-\omega_2|}\le\tilde\Delta(s)=\tilde\Delta(x_c[\mu_t]),\quad\forall~t\ge0,~~\forall~\omega_1\ne\omega_2\in\Omega,
    \end{align*}
    which is equivalent to the positive invariance of $\mathcal C_\nu^{(2)}(\tilde\Delta)$.
\end{proof}

\begin{remark} \label{Remark:FeasibilityBis}
We check that one can  find a non empty open set of parameters of the form
\begin{gather*}
 \big\{ (\kappa,\gamma) : \kappa \in (0,1), \ \gamma \in (0,\tilde\gamma(\kappa)) \big\}
\end{gather*} 
so that \eqref{param1}, \eqref{D_cond}, and \eqref{Equation:LipschitzRegularityAssumption} admit a solution in $D$. Moreover $\tilde\gamma(\kappa) = \mathcal{O}(\kappa)$ and $\|\tilde\Delta\|_\infty = \mathcal{O}(\kappa^{-1})$.
\end{remark}

\begin{proof}
Let
\begin{gather*}
\tilde D(\kappa) := \min \Big( D(\kappa), \frac{BA}{8(\frac{3I}{A}+M +\frac{\kappa I}{A}(I+M))}\Big)
\end{gather*}
where $D(\kappa)$ has been defined  \eqref{Equation:DispersionFormula} and 
\begin{gather*}
\tilde\gamma(\kappa) := \min \Big( \gamma(\kappa), \frac{(\frac{3I}{A}+M +\frac{\kappa I}{A}(I+M))\tilde D(\kappa)}{1+\frac{I}{A}} \Big)
\end{gather*}
where $\gamma(\kappa)$ is defined similarly as in \eqref{Equation:GammaFormula} with $D(\kappa)$ replaced by $\tilde D(\kappa)$. Then \eqref{Equation:LipschitzRegularityAssumption} is satisfied thanks to
\begin{gather*}
E_2(\kappa,\gamma,\tilde D(\kappa)) \leq \frac{4\kappa (\frac{3I}{A}+M +\frac{\kappa I}{A}(I+M))\tilde D(\kappa)}{A} \leq \frac{\kappa B}{2}.
\end{gather*}
In particular 
\begin{gather*}
\tilde \Delta(s) \leq \frac{8\pi \exp(2\pi \kappa(\frac{I}{A}+B))}{\kappa AB}
\end{gather*}
\end{proof}

In the following lemma, we study a regularity with respect to the initial conditions for the metrics $d_1$ and  $d_2$. As in \eqref{distance}, we define $d_1$ by the distance between two measures using their disintegrations: For every $\mu,\mu'\in \mathcal{P}_2(\mathbb{R} \times \Omega)$, we set
\begin{gather*}
d_1(\mu,\mu') := \int W_1(\mu(\cdot,\omega), \mu'(\cdot,\omega)) \, \nu(d\omega).
\end{gather*}
By taking a Markov kernel $p : \mathcal{B}(\mathbb{R}^2) \times \Omega \to [0,1]$ so that $p(dx,dx',\omega) \in \mathcal{P}(\mathbb{R}^2)$ is an optimal coupling between $\mu(\cdot,\omega)$ and $\mu'(\cdot,\omega)$ for a.e. $\omega \in\Omega$ for the $1$-Wasserstein distance, one obtains a coupling $\pi \in\mathcal{P}_2^\nu(\mathbb{R}^2 \times \mathbb{R}^2)$ between $\mu$ and $\mu'$ defined by the formula
\begin{gather*}
\iint \varphi(x,\omega) \psi(x',\omega') \, p(dx,d\omega,dx',d\omega') := \int  \Big( \int \varphi(x,\omega) \psi(x',\omega) \, p(dx,dx',\omega) \Big) \ \nu(d\omega)
\end{gather*}
one gets the estimate
\[
W_1(\mu, \mu') \leq d_1(\mu,\mu').
\]

\begin{lemma} \label{L5.2}
Suppose that $\mu, \mu' \in \mathcal C([0,\infty);\mathcal{P}_2^\nu(\mathbb R^2))$ be global measure-valued solutions to \eqref{A-1} with initial data in $\mathcal P_2^\nu(\mathbb R^2)$. Then, the following assertions hold.
\begin{enumerate}
\item The $d_1$-distance between $\mu$ and $\mu'$ satisfies the local-in-time stability estimate:
\[
d_1(\mu_t,\mu'_t) \leq d_1(\mu_0,\mu'_0) e^{\kappa(M+I)t}, \quad\forall~ t \ge 0.
\]
In particular, for barycenters 
\[ x_c(t) := \int_{\mathbb{R} \times \Omega} x \mu_t(dx,d\omega) \quad \mbox{and} \quad x'_c(t) := \int_{\mathbb{R} \times \Omega} x \mu'_t(dx, d\omega), \]
we have
\[
|x_c(t) - x'_c(t)| \leq |x_c(0) -x'_c(0)| + d_1(\mu_0,\mu'_0) \left( e^{\kappa(M+I)t} - 1 \right), \quad \forall~t \ge 0.
\]
\item The $d_2$-distance between $\mu$ and $\mu'$ satisfies the local-in-time stability estimate:
\[
d_2(\mu_t,\mu'_t) \leq d_2(\mu_0,\mu'_0) e^{\kappa(M+I)t}, \quad \forall~t \ge 0.
\]
\end{enumerate}
\end{lemma}

\begin{proof}
(1) We choose a test function $\varphi\in C_c^1(\mathbb R)$ such that $\mbox{Lip}(\varphi)\le1$ and consider the weak formulation of \eqref{A-1}
\begin{align}\label{E5.2.1}
\begin{aligned} 
\int_\mathbb{R}  \varphi(x)\, \mu_t(dx,\omega)  &= \int_\mathbb{R}  \varphi(x)\, \mu_0(dx,\omega) + \int_0^t  \Big(\int_\mathbb{R}  \frac{d\varphi}{dx}(x,\omega) F[\mu_s](x,\omega)\, \mu_s(dx,\omega) \Big)ds,
\end{aligned}
\end{align}
for $\nu$-almost $\omega\in\Omega$. Recall the Monge-Kantorovich duality of 1-Wasserstein distance on $\mathbb R$
\begin{gather*}
W_1(\eta,\eta') = \sup \left\{ \int_{\mathbb R} \phi(x) \, \eta(dx) - \int_{\mathbb R} \phi(x) \, \eta'(dx) ~:~\phi \in C^1_c(\mathbb{R})~~\mbox{with}~~ \ \mbox{Lip}(\phi) \leq 1 \right\}.
\end{gather*}
We use the form \eqref{E5.2.1} and the duality to obtain
\begin{align}\label{E5.2.2}
\begin{aligned}
&\int_{\mathbb R} \varphi(x)\, \mu_t(dx,\omega) - \int_{\mathbb R} \varphi(x)\, \mu'_t(dx,\omega) \\
& \hspace{1cm} \leq W_1(\mu_0(\cdot,\omega), \mu'_0(\cdot,\omega))  \\
& \hspace{1.4cm} + \int_0^t \bigg| \int_{\mathbb R} F[\mu_s](x,\omega)  \mu_s(dx,\omega) - \int_{\mathbb R} F[\mu'_s](x,\omega)  \mu'_s(dx,\omega) \bigg| ds.
\end{aligned}
\end{align}
Note that the last term in \eqref{E5.2.2} can be estimated as follows:
\begin{align}\label{E5.2.3}
\begin{aligned}
&\bigg| \int_{\mathbb R} F[\mu_s](x,\omega)  \mu_s(dx,\omega) - \int_{\mathbb R} F[\mu'_s](x',\omega)  \mu'_s(dx',\omega) \bigg| \\
& \hspace{1cm} \leq \bigg| \int_{\mathbb R} F[\mu_s](x,\omega)  \mu_s(dx,\omega) - \int_{\mathbb R} F[\mu'_s](x,\omega)  \mu_s(dx,\omega) \bigg|  \\
& \hspace{1.4cm}+ \bigg| \int_{\mathbb R} F[\mu'_s](x,\omega)  \mu_s(dx,\omega) - \int_{\mathbb R} F[\mu'_s](x',\omega)  \mu'_s(dx',\omega) \bigg|\\
& \hspace{1cm} \leq \kappa M W_1(\mu_s,\mu'_s) + \kappa I W_1(\mu_s(\cdot,\omega),\mu'_s(\cdot,\omega))\\
& \hspace{1cm} \leq \kappa M  d_1(\mu_s,\mu'_s) + \kappa I W_1(\mu_s(\cdot,\omega),\mu'_s(\cdot,\omega)),
\end{aligned}
\end{align}
where we use $(\mathcal{A}_1)$, $(\mathcal{A}_2)$ and Monge-Kantorovich duality. For clarity, we use the notation $W_1(\mu_s,\mu_s')$ for the 1-Wasserstein distance on $\mathbb R\times \Omega$ whereas $W_1(\mu_s(\cdot,\omega),\mu'_s(\cdot,\omega))$ on $\mathbb R$. We combine \eqref{E5.2.2} and \eqref{E5.2.3} and take the supremum with respect to $\varphi$ to obtain
\begin{align*}
W_1(\mu_t(\cdot,\omega),\mu'_t(\cdot,\omega)) &\leq W_1(\mu_0(\cdot,\omega), \mu'_0(\cdot,\omega)) + \kappa \int_0^t \Big[ M  d_1(\mu_s,\mu'_s) +  I W_1(\mu_s(\cdot,\omega),\mu'_s(\cdot,\omega)) \Big] ds.
\end{align*}
We integrate in $\omega$ with respect to $\nu$ to find
\begin{gather*}
d_1(\mu_t,\mu'_t) \leq d_1(\mu_0,\mu'_0) + \kappa (M+I) \int_0^t d_1(\mu_s,\mu'_s) \, ds.
\end{gather*}
Then, the Gr\"onwall inequality gives the desired first estimate. Now, it follows from Lemma \ref{L4.1} and the first estimate that 
\begin{align}\label{E5.2.4}
\begin{aligned}
\bigg| \frac{dx'_c}{dt} - \frac{dx_c}{dt} \bigg| &\leq \int_\Omega \bigg| \int_{\mathbb R} F[\mu'_t](x', \omega) \, \mu'_t(dx',\omega) - \int_{\mathbb R}  F[\mu_t](x , \omega) \, \mu_t(dx,\omega) \bigg| \nu(d\omega) \\
&\leq \kappa (M+I) d_1(\mu'_t,\mu_t)  \leq \kappa(M+I) d_1(\mu_0,\mu'_0) e^{\kappa(M+I)t}.
\end{aligned}
\end{align}
Then, we integrate \eqref{E5.2.4} to find the desired stability estimate.

\vspace{.2cm}

\noindent (2) Let $\pi_0:\Omega\to\mathcal P_2(\mathbb R^2)$ be a measurable map such that for $\nu$-almost every $\omega\in\Omega$, the measure $\pi_0(\cdot,\cdot,\omega)$ is an optimal coupling between $\mu_0(\cdot,\omega)$ and $\mu'(\cdot,\omega)$ with respect to the 2-Wasserstein distance, i.e.,
\begin{align*}
    \int_{\mathbb R}\pi_0(x,y,\omega)dy=\mu_0(x,\omega),\quad\int_{\mathbb R}\pi_0(x,y,\omega)dx=\mu'_0(y,\omega),
\end{align*}
and
\begin{align*}
    W_2(\mu_0(\cdot,\omega),\mu'_0(\cdot,\omega))=\left(\int_{\mathbb R^2}|x-y|^2\pi_0(dx,dy,\omega)\right)^{1/2}.
\end{align*}
We define
\begin{align*}
    \Lambda(t):=\left(\int_{\mathbb{R}^2} |X(t,x,\omega) - X'(t,y,\omega) |^2 \, \pi_0(dx,dy,\omega)\nu(d\omega)\right)^{1/2},
\end{align*}
where $X(t,x,\omega)$ and $X'(t,y,\omega)$ are the solutions of the characteristic flow \eqref{char} with respect to $\mu_t$ and $\mu'_t$, respectively. By the definition of $\pi_0$, one has
\begin{align*}
    d_2(\mu_0,\mu'_0)=\Lambda(0)\quad\mbox{and}\quad d_2(\mu_t,\mu'_t)\le\Lambda(t),\quad\forall~t\ge0.
\end{align*}
For simplicity, we write 
\[ X=X(t,x,\omega) \quad \mbox{and} \quad X'=X'(t,y,\omega), \]
leaving the dependence on $x,y$ and $\omega$ implicit, which does not cause any confusion in the remaining estimates. Now, we obtain
\begin{align}
\begin{aligned} \label{E-7}
\frac{1}{2}\frac{d }{dt}\Lambda^2(t) &= \int_{\mathbb{R}^2 \times \Omega} (X-X') \big(F[\mu_t](X,\omega) - F[\mu'_t](X',\omega) \big)\pi_0(dx,dy,\omega)\nu(d\omega)\\
&= \int_{\mathbb{R}^2 \times \Omega} (X-X') \big(F[\mu_t](X,\omega) - F[\mu'_t](X,\omega) \big)\pi_0(dx,dy,\omega)\nu(d\omega)\\
&\hspace{0.4cm} + \int_{\mathbb{R}^2 \times \Omega} (X-X') \big(F[\mu'_t](X,\omega) - F[\mu'_t](X',\omega) \big)\pi_0(dx,dy,\omega)\nu(d\omega)\\
&=: \mathcal{I}_{31} + \mathcal{I}_{32}.
\end{aligned}
\end{align}
Below, we estimate the term ${\mathcal I}_{3i}$ one by one.

\vspace{.2cm}

\noindent $\bullet$ Case A (Estimate of $\mathcal{I}_{31}$):~We use $(\mathcal A_1)$ and Young's inequality to find
\begin{align}
\begin{aligned} \label{E-8}
\mathcal{I}_{31} &\leq \kappa M\int_{\mathbb{R}^2 \times \Omega}  |X-X'|W_1(\mu_t,\mu'_t)\pi_0(dx,dy,\omega)\nu(d\omega) \\ 
&\leq \kappa M \int_{\mathbb{R}^2 \times \Omega} \frac{|X-X'|^2 + W_1(\mu_t,\mu'_t)^2}{2}\pi_0(dx,dy,\omega)\nu(d\omega)\\
&\leq \frac{\kappa M}{2} \left(\Lambda^2(t)+W_2^2(\mu_t,\mu_t')\right)\leq \kappa M \Lambda^2(t).
\end{aligned}
\end{align}

\vspace{.1cm}

\noindent $\bullet$ Case B (Estimate of $\mathcal{I}_{32}$):~By direct calculation with $(\mathcal A_2)$, one has 
\begin{equation} \label{E-9}
\mathcal{I}_{32} \leq \kappa I \int_{\mathbb{R}^2 \times \Omega}|X-X'|^2\pi_0(dx,dy,\omega)\nu(d\omega) = \kappa I \Lambda^2(t).
\end{equation}

\vspace{.2cm}

\noindent In \eqref{E-7}, we combine the estimates \eqref{E-8} and \eqref{E-9} to get 
\[
\frac{1}{2}\frac{d }{dt}\Lambda^2(t) \leq \kappa(M+I) \Lambda^2(t),
\]
which yields
\[
d_2(\mu_t,\mu'_t)\leq \Lambda(t) \leq e^{\kappa(M+I)}\Lambda(0) = e^{\kappa(M+I)}d_2(\mu_0,\mu'_0).
\]
This ends the proof.
\end{proof}
Next,  we study the temporal Lipschitz continuity of $d_2(\mu_t, \mu_{t'})$ with respect to time $t$. 
\begin{lemma}\label{L5.3}
Let $\mu \in\mathcal C([0,\infty);\mathcal{P}_2^\nu(\mathbb{R}^2))$ be a global measure-valued solution to \eqref{A-1} with an initial datum $\mu_0\in\mathcal C_\nu^{(1)}(\Delta)$. Then there exists a constant $F_{\max}>0$ such that 
\begin{equation*}
d_2(\mu_t,\mu_{t'}) \leq F_{\max}|t-t'|, \quad \forall~t,t'\ge0.
\end{equation*}
\end{lemma}
\begin{proof}
    First, we claim that there exists a constant $F_{\text{max}}>0$ such that
    \begin{align}\label{F_max}
        \sup_{(x,\omega)\in\mathbb R\times\Omega}|F[\mu](x,\omega)|\le F_{\text{max}},\quad\forall~\mu\in\mathcal C_\nu^{(1)}(\Delta).
    \end{align}
    As before, we use the notation
    \begin{align*}
        x_c=\int_{\mathbb R\times\Omega}x\mu(dx,d\omega)\quad\mbox{and}\quad\tilde\mu=\delta_{x_c}\otimes\nu.
    \end{align*}
    By $(\mathcal A_1)$, $(\mathcal A_2)$ and \eqref{L3.3.3.1}, one gets
    \begin{align}\label{F_max1}
    \begin{aligned}
        &|F[\mu](x,\omega)-F[\tilde\mu](x_c,\omega)|\\
        &\hspace{1cm}\le|F[\mu](x,\omega)-F[\tilde\mu](x,\omega)|+|F[\tilde\mu](x,\omega)-F[\tilde\mu](x_c,\omega)|\\
        &\hspace{1cm}\le\kappa MW_1(\mu,\tilde\mu)+\kappa I|x-x_c|\le\kappa D(M+I),
    \end{aligned}
    \end{align}
    for all $(x,\omega)\in\Supp(\mu)$. Note that $(\mathcal A_1)$ and $(\mathcal A_2)$ imply that $F[\delta_s\otimes\nu](s,\omega)$ is continuous with respect to $s$ and $\omega$. Since it is periodic in $s$ and $\Omega$ is compact, one has
    \begin{align}\label{F_max2}
    \begin{aligned}
        \sup_{(s,\omega)\in\mathbb R\times\Omega}F[\delta_s\otimes\nu](s,\omega)<\infty.
    \end{aligned}
    \end{align}
    Thus, we combine \eqref{F_max1} and \eqref{F_max2} to deduce the claim \eqref{F_max}.

    \vspace{.1cm}

    Fix $t,t'\ge0$ and $\omega\in\Omega$. By considering the coupling $(X(t,x,\omega),X(t',x,\omega))_\sharp\mu_0$ with \eqref{char} and \eqref{F_max}, we have
    \begin{align*}
        &d_2^2\left(\mu_t(\cdot,\omega),\mu_{t'}(\cdot,\omega)\right)\\
        &\hspace{1cm}\le\int_\Omega\int_{\mathbb R}|X(t,x,\omega)-X(t',x,\omega)|^2\mu_0(dx,\omega)\nu(d\omega)\\
        &\hspace{1cm}\le\int_\Omega\int_{\mathbb R}\bigg|\int_t^{t'}F[\mu_{\tilde t}](X(\tilde t,x,\omega),\omega)ds\bigg|^2\mu_0(dx,\omega)\nu(d\omega)\le F_{\text{max}}^2|t-t'|^2.
    \end{align*}
    This derives the desired result.
\end{proof}

Now, we are ready to provide a proof of Theorem \ref{main} in the next subsection.

\subsection{Proof of Theorem \ref{main}}\label{sec:5.2}
We consider a Poincar\'e section in $\mathcal C_\nu(\Delta,\tilde\Delta)$
\begin{align*}
    \tilde{\mathcal C}_\nu := \left\{ \mu \in \mathcal C_\nu(\Delta,\tilde\Delta) \;:\; x_c[\mu] = 0 \right\},
\end{align*}
and define the first return time map $T:\tilde{\mathcal C}_\nu\to\mathbb R_+$ by
\begin{align*}
    T(\mu_0)=\inf\{t\ge0~:~x_c[\mu_t]=2\pi\},
\end{align*}
where $\mu$ is the solution of \eqref{A-1} with the initial datum $\mu_0$. Since $x_c[\mu_t]$ is strictly increasing in time by Lemma \ref{L4.2}, the well-posedness of $T(\mu_0)$ is guaranteed. Thanks to Proposition \ref{P5.1}, we then define the Poincar\'e map $P:\tilde{\mathcal C}_\nu\to\tilde{\mathcal C}_\nu$ by
\begin{align*} 
    P(\mu_0)=\tau[-2\pi]_\sharp\mu_{T(\mu_0)},
\end{align*}
where $\tau[-2\pi]:\mathbb R\to\mathbb R$ is the translation $\tau[-2\pi](x):=x-2\pi$. Since $x_c[\mu_{T(\mu_0)}] = 2\pi$, the shifted measure $\tau[-2\pi]_\sharp\mu_{T(\mu_0)}$ lies in $\tilde{\mathcal C}_\nu$, and hence $P$ is well-defined. We equip $\mathcal P_2^\nu(\mathbb R^2)$ with the metric $d_2$ given in \eqref{distance}. We recall that $\mathcal{X}_\nu$ and $\mathcal P_2^\nu(\mathbb R^2)$ have been identified. The space $(\mathcal P_2^\nu(\mathbb R^2), d_2)$ becomes a CAT$(0)$ geodesic space. As $\tilde{\mathcal C}_\nu$ is a closed, convex subset, we can apply the Schauder fixed point theorem given by Proposition \ref{Sch_thm}. Note that if the Poincar\'e map $P$ admits a fixed point, then the proof of Theorem \ref{main} is complete. Therefore, to apply the fixed point theorem, it suffices to show that the Poincar\'e map $P$ is continuous. To this end, we will show that $P$ is Lipschitz continuous.

First, we show that the return time $T(\cdot)$ is Lipschitz continuous. Consider $\mu_0, \mu_0' \in \tilde{\mathcal C}_\nu$ and assume $T(\mu_0') \leq T(\mu_0)$ without loss of generality. Denote their barycenters by 
\[
x_c(t) := \int_{\mathbb{R}\times \Omega}x \mu_t(dx,d\omega) \quad \text{and} \quad x_c'(t) := \int_{\mathbb{R}\times \Omega}x \mu_t'(dx,d\omega).
\]
From Lemma \ref{L4.2} and $x_c(0)=0$, one has
\begin{align*}
	2\pi= x_c \circ T(\mu_0) = \int_0^{T(\mu_0)} F[\mu_t](x,\omega) \, \mu_t(dx,d\omega) \geq \big(A-((\kappa+1)I+\kappa M)D-I\gamma\big)T(\mu_0),
\end{align*}
hence,
\begin{align*}
    \frac{2\pi}{A-((\kappa+1)I+ \kappa M)D-I\gamma}\ge T(\mu_0).
\end{align*}
Now, it follows from Lemma \ref{L5.2} with $x_c'(0)=0$ that for $t\leq T(\mu_0)$,
\begin{align}\label{E5.3.2}
\begin{aligned}
x_c'(t) &\leq x_c(t) + d_1(\mu_0,\mu_0') \left(e^{\kappa(M+I)t} - 1\right) \\
&\leq x_c \circ T(\mu'_0)  - \int_t^{T(\mu'_0)} F[\mu_t](x,\omega) \, \mu_t(dx,d\omega) + d_1(\mu_0,\mu_0') \left(e^{\kappa(M+I)t} - 1\right) \\
&\leq 2\pi - \big(A-((\kappa+1)I+\kappa M)D-I\gamma\big)(T(\mu'_0)-t) + d_1(\mu_0,\mu'_0) \left(e^{\kappa(M+I)t}-1\right),
\end{aligned}
\end{align}
where the last inequality comes from integrating
\[
 F[\mu_t](x,\omega)  \geq A-((\kappa+1)I+\kappa M)D-I\gamma
\]
from $t$ to $T(\mu_0)$. We set $t = T(\mu_0')$ in \eqref{E5.3.2} to obtain 
\begin{align*}
    &\big(A-((\kappa+1)I+\kappa M)D-I\gamma\big)(T(\mu_0)-T(\mu_0'))\\
    &\hspace{1cm}\le d_1(\mu_0,\mu_0')\left(e^{\kappa(M+I)T(\mu_0')}-1\right)\\
    &\hspace{1cm}\le d_1(\mu_0,\mu_0')\left(\exp\left(\frac{2\pi\kappa(M+I)}{A-((\kappa+1)I+\kappa M)D-I\gamma}\right)-1\right).
\end{align*}
Thus, we have verified that the return time is Lipschitz continuous: 
\begin{align}\label{E5.3.3}
\begin{aligned}
&|T(\mu_0) - T(\mu_0')|\\
& \hspace{1cm}\le \frac{d_1(\mu_0,\mu_0')}{A-((\kappa+1)I+\kappa M)D-I\gamma}\left(\exp\left(\frac{2\pi\kappa(M+I)}{A-((\kappa+1)I+ \kappa M)D-I\gamma}\right)-1\right).
\end{aligned}
\end{align}

\vspace{.2cm}

Next, we show that the Poincar\'e map  $P$ is Lipschitz continuous with respect to the metric $d_2$. From Lemma \ref{L5.2} and Lemma \ref{L5.3}, we observe that
\begin{align*}
\begin{aligned}
&d_2\left(P(\mu_0),P(\mu_0')\right)= d_2\big(\mu_{T(\mu_0)},\mu'_{T(\mu_0')}\big) \\
&\hspace{1cm} \leq d_2\big(\mu_{T(\mu_0)},\mu'_{T(\mu_0)}\big)+d_2\big(\mu'_{T(\mu_0)},\mu'_{T(\mu'_0)}\big) \\
&\hspace{1cm}\leq e^{\kappa(M+I)T(\mu_0)} d_2(\mu_0,\mu'_0) + F_{\max} |T(\mu_0) - T(\mu_0')|\\
& \hspace{1cm}\leq \exp\left(\frac{2\pi\kappa(M+I)}{A-((\kappa+1)I+\kappa M)D-I\gamma}\right)d_2(\mu_0,\mu'_0) + F_{\max} |T(\mu_0) - T(\mu_0')|.
\end{aligned}
\end{align*}
Using the fact $W_1(\mu_0,\mu_0') \leq W_2(\mu_0,\mu_0')$ and the Lipschitz continuity \eqref{E5.3.3} of the return time $T$, we estimate as follows:
\begin{align*}
d_2\big(P(\mu_0),P(\mu'_0)\big) & \leq \left(  e^{C} + \frac{F_{\max}}{A-((\kappa+1)I+\kappa M)D-I\gamma} (e^{C}-1) \right) d_2\big(\mu_0,\mu'_0\big),
\end{align*}
where the constant $C$ is defined by
\begin{align*}
    C=\frac{2\pi\kappa(M+I)}{A-((\kappa+1)I+ \kappa M)D-I\gamma}.
\end{align*}
Therefore, the Poincar\'e map $P$ is continuous and the proof is complete.

\section{Further regularity of a periodic solution}\label{sec:6}
\setcounter{equation}{0}
In this section, we study the regularity of a periodic solution whose existence is guaranteed by Theorem \ref{main}. In Section \ref{sec:5}, we have constructed a periodic measure-valued solution to \eqref{A-1}. If we regard \eqref{A-1} as a Vlasov-type equation obtained as the mean-field limit of a interacting particle system such as Winfree and Kuramoto models, then we note that the equation \eqref{A-1} incorporates the empirical measures which are solutions to the particle system. However, the existence of periodic solutions in particle systems is already obtained in \cite{O-K-T}, in an elementary way without an optimal transport. A natural question is whether our periodic measure-valued solution is distinct from the empirical measure, i.e. the particle solution. We address this question in this section. In fact, if $\omega$-marginal $\nu$ has a density or at least it is not a Dirac mass, it is easy to see that the periodic solution $\mu$ cannot be the particle solution of the form:
\begin{equation} \label{F-1}
\mu_t = \frac{1}{N}\sum_{i=1}^N \delta_{(x_i(t), \omega_i(t))}, \quad \text{for some} ~N \in \mathbb{N}.
\end{equation}
This can be seen as follows. If $\mu$ is an empirical measure such as \eqref{F-1}, then we have 
\[ \nu(d\omega) = \int_{\mathbb R_x} \mu(dx, d\omega) = \frac{1}{N}\sum_{i=1}^N\delta_{\omega_i(t)}. \]
 This is contradictory to the assumption on $\nu$. Thus, it is still desirable to investigate the further regularity of the periodic measure-valued solution $\mu$.

\subsection{Periodic graph measures} \label{sec:6.1}
In this subsection, we impose the following additional regularity property and a new twist property for the vector field $F[\mu]$: 

\vspace{0.2cm}
\begin{itemize}
\item ($\mathcal{A}_6$)(Additional regularity hypothesis): for every $\mu,\mu' \in\mathcal{P}_2^\nu(\mathbb{R}^2)$,
\begin{gather*}
\forall\, (x,\omega) \in\mathbb{R} \times \Omega, \ |\partial_\omega F[\mu](x,\omega) - \partial_\omega F[\mu'](x,\omega) | \leq M W_1(\mu,\mu'), \quad
\|\partial^2_{x\omega}F[\mu] \|_\infty \leq \kappa I,
\end{gather*}
\item ($\mathcal{A}_7$)(Twist inequality hypothesis):~There exists a constant $Q>0$ such that 
\begin{equation}\label{twist}
\forall\, x \in\mathbb{R}, \ \partial_\omega F[\delta_x \otimes \nu](x,\omega_c) \geq Q.
\end{equation}
\end{itemize}

\vspace{0.2cm}
Notice that the Kuramoto and Winfree models satisfy the twist property. For a given point $(a,\omega) \in \mathrm{supp}(\mu_0)$, let $X(t,a,\omega) = X[\mu_t](a,\omega)$ be a global solution to the characteristic flow \eqref{char}. First, we study the monotonicity property of the flow with respect to $\omega$. To unsure a global twist property, we assume 
\begin{gather}
DM + \kappa I < Q. \label{Equation:GlobalTwistProperty}
\end{gather}

\begin{lemma}\label{L6.1}
\emph{(Monotonicity property)}
Let $\mu \in {\mathcal C}([0, \infty); \mathcal{P}_2^\nu(\mathbb{R}^2))$ be a global measure-valued solution to \eqref{A-1}. Then, we have a monotonicity of $X(t, x, \omega)$ with respect to $\omega$ for fixed $(t,x)$: more precisely for any $(x,\omega_1),(x,\omega_2)\in\Supp(\mu_0)$,
\[
\omega_1 > \omega_2 \quad \Longrightarrow \quad X(t,x,\omega_1) > X(t,x,\omega_2), \quad \forall~t > 0.
\]
\end{lemma}
\begin{proof} For $(x,\omega) \in \mathrm{supp}(\mu_0)$, we recall the equation for a forward trajectory:
\begin{equation} 
\begin{cases} \label{CF}
\displaystyle \partial_t X(t, x,\omega) = F[\mu_t](X(t, x,\omega),\omega), \quad \forall~t > 0, \\
\displaystyle  X(0,x,\omega) = x.
\end{cases}
\end{equation}
We differentiate \eqref{CF} with respect to $\omega$ to obtain
\begin{align}
\begin{aligned} \label{F-3}
\partial_\omega \partial_t X(t, a,\omega) &= \frac{d}{d\omega} F[\mu_t](X(t, a,\omega),\omega)\\
&=\partial_x F[\mu_t](X(t,a,\omega),\omega)\partial_\omega X + \partial_\omega F[\mu_t](X(t,a,\omega),\omega).
\end{aligned}
\end{align}
Now, we integrate \eqref{F-3} from time $0$ to $t$ using $\partial_\omega X(0,x,\omega) = 0$ to obtain
\[
\partial_\omega X(t,x,\omega) = \int_0^t \mathrm{exp}\bigg(\int_s^t\partial_xF[\mu_u](X(u,x,\omega),\omega)du\bigg)\partial_\omega F[\mu_s](X(s,x,\omega),\omega)ds.
\]
Using ($\mathcal{A}_6$) and $X=X(t,x,\omega)$ for simplification, one obtains
\begin{align*}
| \partial_\omega F[\mu_t](X,\omega) &- \partial_\omega F[\tilde\mu_t](x_c(t),\omega) | \\
&\leq | \partial_\omega F[\mu_t](X,\omega) - \partial_\omega F[\tilde\mu_t](X,\omega) | + | \partial_\omega F[\tilde\mu_t](X,\omega) - \partial_\omega F[\tilde\mu_t](x_c(t),\omega) | \\
&\leq M W_1(\mu_t,\tilde \mu_t) + \kappa I |X-x_c(t)| \leq DM +\kappa I.
\end{align*}
Then \eqref{twist} and \eqref{Equation:GlobalTwistProperty} imply
\[
\partial_\omega X(t,x,\omega) > 0.
\]
This yields the desired estimate.
\end{proof}
Lemma \ref{L6.1} shows that for a fixed $x\in\mathbb R$, if the Borel measure of $\Supp(\mu_0)\cap(\{x\}\times\Omega)$ is nonzero, then the mass initially concentrated on $x$ immediately spreads out. In other words, a solution $\mu$ to \eqref{A-1} cannot take the form
\begin{equation*}\label{form}
\int_{\mathbb{R}_\omega} \mu(\cdot, d\omega) = \frac{1}{N}\sum_{i=1}^N \delta_{x_i(t)}, \quad \mbox{for some}~~N \in  \mathbb{N}.
\end{equation*}
even for sufficiently small $t>0$. This means that the projection of $\mu$ on the $x$-axis does not exhibit particle-like behavior. Interestingly, this indicates that the dynamics in the $x$-direction are governed by the structure in the $\omega$-direction. In what follows, we focus on measures supported on the graph $(G_t(\omega), \omega)$ of $\omega$ where $G_t :\Omega \rightarrow \mathbb{R}$ is a  $\mathcal{C}^1$ function supported on $\Omega$. More generally, we say that $\mu$ is a Borel graph measure if there is a Borel function $G : \Omega \to \mathbb{R}$ such that for any Borel test function $\varphi : \mathbb{R} \times \Omega \to \mathbb{R}^+$, we have
\[
\int_{\mathbb R \times \Omega} \varphi(x,\omega) \mu(dx,d\omega) = \int_{\Omega}\varphi(G(\omega),\omega)d\nu(\omega).
\]
We will write $\mu =( G \otimes \Id)_\sharp \nu$. Moreover, we say that a graph measure $\mu$ is continuous/$C^1$ if $G:\Omega\to\mathbb R$ is continuous/$C^1$. In the $C^1$ setting, we will assume to simplify that $\Omega$ is an interval. Our analysis is nevertheless also valid for a set $\Omega$ equal to a finite disjoint union of  intervals.

\begin{lemma}\label{L6.2}
Let $\mu\in\mathcal C([0,\infty);\mathcal P_2^\nu(\mathbb R^2))$ be a global measure-valued solution to \eqref{A-1} with  initial condition $\mu_0$. If $\mu_0$  is a  Borel, respectively continous or $C^1$, graph measure, then, $\mu_t$ is a Borel, respectiveley continuous or $C^1$, graph measure for all $t\ge0$.
\end{lemma}

\begin{proof}

Recall that the characteristic flow in $\mathbb{R}^2$ is given by the map $(x,\omega) \mapsto (X[\mu_t](t,x,\omega),\omega)$ and that $\mu_t = (X[\mu_t](t) \otimes \Id)_\sharp \mu_0$. Assume $\mu_0 = (G_0 \otimes \Id)_\sharp \nu$. Then for any Borel test function $\varphi : \mathbb{R} \times \Omega \to \mathbb{R}^+$
\begin{align*}
\int \varphi(x,\omega) \, \mu_t(dx,d\omega) &= \int \varphi(X[\mu_t](t,x,\omega),\omega) \, \mu_0(dx,d\omega) \\
&= \int \varphi( X[\mu_t](t,G_0(\omega),\omega),\omega) \, \nu(d\omega) \\
&= \int \varphi(G_t(\omega),\omega) \, \nu(d\omega),
\end{align*}
where $G_t(\omega) = X[\mu_t](t,G_0(\omega),\omega)$. We have obtained $\mu_t = (G_t \otimes \Id)_\sharp \nu$.

\end{proof}

By Proposition \ref{P5.1} and Lemma \ref{L6.2}, we can define a third invariant set $\mathcal G(\Delta,\tilde\Delta)$ as a subset of $\mathcal C_\nu(\Delta,\tilde\Delta)$ constructed in \eqref{C_nu_def}:
\begin{gather*}
    \mathcal G :=\big\{\mu\in\mathcal C_\nu(\Delta,\tilde\Delta):\mu\mbox{ is a Borel graph measure}\big\}.
\end{gather*}
Actually, if $\mu = (G \otimes \Id)_\sharp \nu \in \mathcal G(\Delta,\tilde\Delta)$ then its disintegrations $(\mu( \cdot, \omega))_{\omega\in\Omega}$ are Dirac measures
\begin{gather*}
\mu(dx,\omega) = \delta_{G(\omega)}(dx).
\end{gather*}
Moreover, the $2$-Wasserstein distance between these disintegrations $\mu(\cdot,\omega_1)$ and $\mu(\cdot,\omega_2)$ becomes
\[
W_2(\mu(\cdot,\omega_1),\mu(\cdot,\omega_2)) = |G(\omega_1) - G(\omega_2)|.
\]
In particular, the Lipschitz property in $\mathcal{C}_\nu(\Delta,\tilde \Delta)$:
\[
W_2(\mu(\cdot,\omega_1),\mu(\cdot,\omega_2)) \leq \tilde{\Delta}(x_c[\mu])|\omega_1 - \omega_2|,
\]
is now translated to the Lipschitz boundedness of $G$:
\begin{equation}\label{C6.1.0}
|G(\omega_1) - G(\omega_2)| \leq \tilde{\Delta}(x_c[G])|\omega_1 - \omega_2|.
\end{equation}
Thus
\begin{gather*}
    \mathcal G :=\big\{\mu \in \mathcal C_\nu(\Delta,\tilde\Delta) : \mu \ \text{is $\Delta(x_c[G])$-Lipschiz} \ \big\},
\end{gather*}
where 
\[
x_c[G] := x_c[(G \otimes \Id)_\sharp \nu] = \int_{\mathbb{R} \times \Omega}x \, (G \otimes \Id)_\sharp \nu(dx, d\omega) = \int_{\Omega}G(\omega)\nu(d\omega).
\]

Thus, we can apply the proof of Theorem \ref{main} to obtain the following corollary.
\begin{corollary}
Assume that the functional $F$ satisfies $(\mathcal A_1)-(\mathcal A_5)$, the additional regularity $(\mathcal{A}_6)$, and the twist property \eqref{twist} in $(\mathcal{A}_7)$. Then the following assertions hold:
\begin{enumerate}
\item The set $\mathcal{G}(\Delta,\tilde\Delta)$ is positively invariant along the flow of \eqref{A-1}, i.e., the solution $\mu$ to \eqref{A-1} with the initial condition $\mu_0\in\mathcal G$ satisfies
\begin{gather*}
\mu_t \in \mathcal{G}(\Delta,\tilde\Delta), \quad \forall~t \geq 0. 
\end{gather*}
\item There exists $\mu_* \in \mathcal{G}$ such that the solution $\mu$ with the initial datum $\mu_0=\mu_*$ is periodic in the following sense:  
\[
\mu_{t+T_*} = \tau[2\pi]_\sharp(\mu_t), \quad \forall~t \geq 0,
\]
where  $\tau[2\pi]:\mathbb R\to\mathbb R$ is defined by $\tau(x):=x+2\pi$ and
\[T_* =  \inf \{ t \geq0 : x_c[\mu_t] = x_c[\mu_0]+2\pi \}.\] 
\end{enumerate}
\end{corollary}

\begin{proof}
The assertion {\it(1)} follows from the discussion preceding the theorem. We now prove {\it(2)}. Likewise in Theorem \ref{main}, we construct a Poincar\'e map on the section
\[
\mathcal{G}_0(\Delta,\tilde \Delta) := \{\mu \in \mathcal{G}(\Delta,\tilde\Delta) : x_c[G] = 0 \}.
\]
In order to apply the fixed point theorem, we do not need to work with a geodesic space, as the set $\mathcal G_0(\Delta,\tilde\Delta)$ can be identified with
\begin{align*}
    \left\{G\in\mathcal C(\Omega;\mathbb R):x_c[G]=0,\sup_{\omega_1,\omega_2\in\Omega}|G(\omega_1)-G(\omega_2)|\le\Delta(0),~~\|G\|_{\text{Lip}}\le\tilde\Delta(0)\right\}.
\end{align*}
We denote a graph measure with respect to the function $G$ by $\mu_G$. Since the proof is very lengthy, we split its proof into several steps.

\vspace{.2cm}

\noindent $\bullet$~Step A (Convexity of $\mathcal{G}_0(\Delta,\tilde \Delta)$):~For $\mu_{G_1}, \mu_{G_2} \in \mathcal{G}_0(\Delta,\tilde \Delta)$ and $u \in (0,1)$, we note that 
\begin{align*}
    x_c[uG_1 + (1-u)G_2] &= \int_{\mathbb{R}}(uG_1(\omega) + (1-u)G_2(\omega))d\nu(\omega)\\
    &= ux_c[G_1] + (1-u)x_c[G_2] = 0.
\end{align*}
Next, it follows from \eqref{C6.1.0} that 
\[ \|G_1\|_{\mathrm{Lip}}, ~\|G_2\|_{\mathrm{Lip}} \leq \tilde{\Delta}(0). \]
Thus, we have
\[
\|uG_1 + (1-u)G_2\|_{\mathrm{Lip}} \leq u \|G_1\|_{\mathrm{Lip}} + (1-u)\|G_2\|_{\mathrm{Lip}} \leq \tilde{\Delta}(0).
\]
This yields 
\[ uG_1 + (1-u)G_2 \in \mathcal{G}_0(\Delta,\tilde \Delta). \]
The last condition can be shown in a similar manner. This verifies the convexity of $\mathcal{G}_0(\Delta,\tilde \Delta)$ with respect to the Lipschitz norm $\|\cdot\|_{\mathrm{Lip}}$. 

\vspace{0.2cm}

\noindent $\bullet$~Step B (Compactness of $\mathcal{G}_0(\Delta,\tilde \Delta)$): Let $\{G_n\}_{n=1}^\infty$ be a sequence in $\mathcal{G}_0(\Delta,\tilde \Delta)$. Since the center $x_c$ vanishes for any $G_n \in \tilde{S}_\gamma$ we have 
\begin{align*}
\|G_n(\omega) \|_\infty &= \|G_n(\omega) - x_c[G_n]\|_\infty \\
&= \bigg\|\int_\Omega (G_n(\omega) - G_n(\omega^*))d\nu(\omega^*)\bigg\|_\infty \leq \|G_n\|_{\mathrm{Lip}} \cdot \gamma \leq \tilde{\Delta}(0) \gamma. 
\end{align*}
Hence the sequence $\{G_n\}_{n=1}^\infty$ is uniformly bounded. Moreover the sequence $\{G_n\}_{n=1}^\infty$ is equicontinuous, as it is uniformly Lipschitz. Thus, it follows from the Arzela-Ascoli theorem that there exists a continuous function $\mathfrak{g}$ to which $G_n$ uniformly converges up to a subsequence as $n \rightarrow \infty$. Furthermore, we pass to the limit $n\rightarrow \infty$ in
\[
\bigg|\frac{G_n(\omega) - G_n(\omega')}{\omega - \omega'}\bigg| \leq \tilde{\Delta}(0), \quad |G(\omega_1) - G(\omega_2)| \leq \Delta(0), \quad \text{for any} \quad \omega \neq \omega' \in \Omega
\]
to see that the function $\mathfrak{g}$ is also Lipschitz. Therefore the set $\mathcal{G}_0(\Delta,\tilde \Delta)$ is compact. 

\vspace{0.2cm}

\noindent $\bullet$~Step C (Continuity of $\tilde{P}$): Let $G_0 \in \mathcal G_0(\Delta,\tilde \Delta)$, $\mu_0 = (G_0 \otimes \Id)_\sharp \nu$, and $G_t \in C(\Omega; \mathbb{R})$ such that $\mu_t = (G_t \otimes \Id)_\sharp \nu$ for all $t\geq0$. Define the Poincar\'e map $\tilde P:\mathcal G_0(\Delta,\tilde \Delta)\to \mathcal G_0(\Delta,\tilde \Delta)$ by
\begin{align*}
\tilde P(G_0):= G_{T(G)},
\end{align*}
where $T(G)$ is the first return time map defined by
\begin{align*}
    T(G):=\inf\{t\ge0:x_c[G_t]=2\pi\}.
\end{align*}
We estimate that for $G,G' \in \mathcal{G}_0(\Delta,\tilde \Delta)$,
\begin{align}\label{E6.1.2}
\begin{aligned}
&|\tilde{P}(G)(\omega) - \tilde{P}(G')(\omega)|= |G_{T(G)}(\omega) - G'_{T(G')}(\omega)| \\
&\hspace{1.2cm}\leq |G_{T(G)}(\omega) - G_{T(G')}(\omega)| + |G_{T(G')} (\omega)- G'_{T(G')(\omega)}|\\
&\hspace{1.2cm}\leq F_{\max}|T(G) - T(G')| + |G_{T(G')}(\omega) - G'_{T(G')}(\omega)|,
\end{aligned}
\end{align}
where $F_{\max}$ is introduced in Lemma \ref{L5.3}. For fixed $\omega\in\Omega$, one has
\begin{align*}
\begin{aligned}
& \frac{1}{2}\frac{d}{dt}|G_t(\omega) - G'_t(\omega)|^2  = (G_t(\omega) - G'_t(\omega))(\frac{d}{dt}G_t(\omega) - \frac{d}{dt}G'_t(\omega))\\
& \hspace{0.8cm} =(G_t(\omega) - G'_t(\omega))\bigg(F[\mu_{G_t}](G_t(\omega),\omega) - F[\mu_{G'_t}](G'_t(\omega),\omega)\bigg)\\
& \hspace{0.8cm} =(G_t (\omega)- G'_t(\omega))\bigg(F[\mu_{G_t}](G_t(\omega),\omega) - F[\mu_{G'_t}](G_t(\omega),\omega)\bigg)\\
& \hspace{1.2cm}+(G_t(\omega) - G'_t(\omega))\bigg(F[\mu_{G'_t}](G_t(\omega),\omega) - F[\mu_{G'_t}](G'_t(\omega),\omega)\bigg)\\
& \hspace{0.8cm} \leq \kappa M\|G_t-G'_t\|_\infty\cdot|G_t(\omega)-G'_t(\omega)|+\kappa I|G_t(\omega)-G'_t(\omega)|^2,
\end{aligned}
\end{align*}
where we used $(\mathcal A_1),(\mathcal A_2)$ and 
\begin{align*}
    W_1(\mu_{G_t},\mu_{G'_t})\le\|G_t-G'_t\|_\infty.
\end{align*}
This yields
\begin{align*}
    \frac{d}{dt}\|G_t-G'_t\|_\infty\le\kappa(M+I)\|G_t-G'_t\|_\infty,
\end{align*}
hence,
\begin{equation}\label{E6.1.3}
\|G_t-G'_t\|_\infty\le e^{\kappa(M+I)t}\|G_0-G'_0\|_\infty.
\end{equation}
Now, we apply \eqref{E5.3.3} and \eqref{E6.1.3} to \eqref{E6.1.2} to obtain
\begin{align*}
    \|\tilde P(G)-\tilde P(G')\|_\infty\lesssim\|G_0-G'_0\|_\infty.
\end{align*}
Thus, the Poincar\'e map $\tilde{P}$ is continuous, and then we use the standard Schauder fixed point theorem to obtain a periodic function in $\mathcal G_0(\Delta,\tilde\Delta)$ which corresponds to a periodic continuous graph measure solution to \eqref{A-1}.
\end{proof}

\section{Conclusion}\label{sec:7}
In this paper, we have provided an existence of a time-periodic measure solutions to the parameterized nonlocal continuity equation using optimal transport and metric geometry. The parameterzed nonlocal continuity equation can be regarded as the infinite class of the nonlocal continuity equation parametrized by real parameters. For the proposed parameterized nonlocal continuity equation, we provide a global existence of periodic measure-valued solutions using Schauder's fixed point argument under a suitable sufficient framework which was formulated in terms of system parameters and initial data. Moreover, we also introduce the concept of graph measure solution which is supported on a graph. For the existence of this special type of measure-valued solution, we also provide sufficient frameworks leading to an existence of measure-valued solutions and graph measure solution. Of course, there are several issues that we did not investigate in this work. For example, since our existence proof is based on fixed point argument, we do not know whether the constructed measure-valued solution and graph measure solutions are unique or not (see Theorem \ref{main}). Second, we expect that a periodic measure-valued solution is stable in some neighborhood. The extension of the current results toward those problems might be an interesting and challenging open problem for future work.

\end{document}